\documentclass[12pt, reqno]{amsart}

\usepackage{amsmath,amsfonts,amsthm,amssymb,color,hyperref}

\usepackage{mathrsfs}

\usepackage[T1]{fontenc}

\usepackage{graphicx}
\usepackage{showlabels}
\usepackage{subfigure} 

\makeatletter
     \def\section{\@startsection{section}{1}%
     \z@{.7\linespacing\@plus\linespacing}{.5\linespacing}%
     {\bfseries
     \centering
     }}
     \def\@secnumfont{\bfseries}
     \makeatother

\usepackage{graphicx}

\newcommand{\R}{\mathbb R}
\newcommand{\RR}{\mathbb R}
\newcommand{\N}{\mathbb N}

\newcommand{\E}{\mathbb E}

\newcommand{\HH}{\mathfrak H}

\newtheorem{theorem}{Theorem}[section]
\newtheorem{lemma}[theorem]{Lemma}
\newtheorem{proposition}[theorem]{Proposition}

\theoremstyle{definition}

\theoremstyle{remark}
\newtheorem{remark}{Remark}
\numberwithin{equation}{section}
\begin{document}
\title[A CLT for the stochastic wave equation with fractional noise]{A Central Limit Theorem for the stochastic wave equation with fractional noise}

\author[F. Delgado-Vences]{Francisco Delgado-Vences}
\address{Conacyt  Research Fellow - Universidad Nacional Aut\'onoma de M\'exico. Instituto de Matem\'aticas, Oaxaca, M\'exico}
\email{delgado@im.unam.mx}

\author[D. Nualart]{David Nualart} \thanks {D. Nualart is supported by NSF Grant DMS 1811181.}
\address{University of Kansas, Department of Mathematics, USA}
\email{nualart@ku.edu}

\author[G. Zheng]{Guangqu Zheng}
\address{University of Kansas, Department of Mathematics, USA}
\email{zhengguangqu@gmail.com}

\begin{abstract}
We study the    one-dimensional stochastic wave  equation driven by a   Gaussian  multiplicative  noise, which is white in time and  has the covariance of a fractional Brownian motion with Hurst parameter $H\in  [1/2,1)$ in the spatial variable. We show that the  normalized spatial average of the solution over $[-R,R]$ converges in total variation distance to a   normal distribution, as $R$ tends to infinity.   We also provide a functional central limit theorem.
\end{abstract}

\maketitle



\allowdisplaybreaks

\section{Introduction}

We consider the one-dimensional stochastic wave equation  
\begin{equation}
\label{eq:heat-equation}
\frac{\partial^2 u}{\partial t^2} =  \frac {\partial ^2 u} {\partial  x^2} + \sigma(u)  \frac { \partial  ^2 W} {\partial t \partial x},
\end{equation}
on $\R_+\times \R$, where $W(t,x)$ is a Gaussian random field that is  a Brownian motion in time and  behaves as a fractional Brownian motion with Hurst parameter $H \in [1/2,1)$ in the spatial variable.
For $H=1/2$, the random field $W$ is just a two-parameter Wiener process on $\R_+\times \R$. We assume   $u(0,x)=1$, $\frac {\partial  } {\partial t} u(0,x)=0$ and
  $\sigma$ is a Lipschitz function with Lipschitz constant $L\in(0,\infty)$.

It is well-known (see, for instance,  \cite{Dalang})
that  equation \eqref{eq:heat-equation} has a unique  \emph{mild solution}, which is adapted to the filtration generated by $W$, such that $\sup\big\{  \E \big[  \vert u(t,x)\vert^2\big]\,:\,   x\in \RR, t\in[0,T] \big\}      < \infty$ and  
\begin{equation}\label{eq: mild}
u(t,x) = 1+  \frac 12 \int_0^t \int_{\RR}  \mathbf{1}_{\{ |x-y| \le t-s\}} \sigma(u(s,y))W(ds,dy)\,, 
\end{equation}
where   the above stochastic integral is defined in the sense of It\^o-Walsh. 

In this paper, we are interested in the asymptotic behavior as $R$ tends to infinity of the spatial averages 
\begin{equation} \label{1}
\int_ {-R}^R u(t,x)dx,
\end{equation}
where $t>0$ is fixed and $u(t,x)$ is the solution to \eqref{eq:heat-equation}.  We remark that, for each fixed $t>0$, the process $\{ u(t,x), x\in \RR\}$ is \emph{strictly stationary}\footnote{ To see the strict stationarity, we fix $y\in\R$ and put $v(t,x) = u(t,x+y)$:  It is clear that $v$ solves the stochastic heat equation \eqref{eq:heat-equation} driven by  the shifted noise $\{ W(t,x+y), t\in\R_+, x\in\R\} $, which has stationary increments in the spatial variable. },  meaning that the finite-dimension distributions of the process $\{u(t,x+y),x\in \RR\}$ do not depend on $y$.   Furthermore, $u(t,x)$ is measurable with respect to the $\sigma$-field generated by the random variables
$\{W(s,z): |x-z| \le t-s\}$.   As a consequence, 
\begin{enumerate}
\item for $H=1/2$,  the  random variables $u(t,x)$ and $u(t,y)$ are independent if $|x-y| > 2t$; 

\item for $H\in (1/2,1)$, $u(t,x)$ and $u(t,y)$ have a correlation that decays like $|x-y-2t| ^{2H-2}$ when $|x-y|\to +\infty$, which is a consequence of Gebelein's inequality (see, for instance,  \cite{V}). 
\end{enumerate} 
Therefore, we expect the Gaussian fluctuation of the   spatial averages \eqref{1}. 

\medskip

Our first goal is to apply the  methodology of Malliavin-Stein to provide a quantitative central limit theorem for \eqref{1}, which will be described in total variation distance.

Define the normalized averages by
 \begin{equation}
\label{eq:quantity-of-interest}
F_R(t):= \frac{1}{\sigma_R}\left(\int_ {-R}^R u(t,x)dx- 2R\right),
\end{equation}
where  $u(t,x)$ is the solution to \eqref{eq:heat-equation} and $\sigma_R^2 ={\displaystyle  {\rm Var}\left(\int_{-R}^R u(t,x)dx\right) }$.

   To avoid triviality, throughout  this paper, we assume that   $\sigma(1)\not =0$, which  guarantees that $\sigma_R>0$ for all $R>0$ and also that  $\sigma_R$  is of order $R^{H}$;  see Lemma \ref{lema1} and Propositions \ref{pro:covariance1}, \ref{pro:covariance2}  below.

 \medskip
 
 Our first result is the following quantitative central limit theorem.
\begin{theorem}
\label{thm:TV-distance}  Let $d_{\rm TV}$ denote the total variation distance $($see \eqref{TVDIST}$)$ and let $Z\sim  \mathcal{N}(0,1)$. For any fixed $t> 0$,  there exists a constant $C= C_{t,H,\sigma}$, depending on $t, H$ and $\sigma$, such that 
$$
d_{\rm TV}\left( F_R(t),Z\right) \leq C R^{H-1} \,.
$$
\end{theorem}

Our second objective is to provide  the functional version of Theorem \ref{thm:TV-distance}. 
 
\begin{theorem}
\label{thm:functional-CLT}
For any $s>0$,  we set  $\eta(s)= \E\big[  \sigma(u(s,y)) \big]$ and $\xi(s)= \E\big[ \sigma^2(u(s,y)) \big]$, which do not depend on $y$ due to the stationarity.
Then, for any $T>0$, as $R\to+\infty$, 
\begin{itemize}
\item[(i)] if $H=  1/2$, then
$$
\left\{ \frac{1}{\sqrt{R}}\left(\int_{-R}^R u(t,x)dx - 2R\right) \right\}_{t\in [0,T]} \Rightarrow \left\{ \sqrt{2} \int_0^t  (t-s) \sqrt{\xi(s)}  dB_s\right\}_{t\in [0,T]};
$$
\item[(ii)] if $H\in (1/2,1)$, then
$$
\left\{  R^{-H} \left(\int_{-R}^R u(t,x)dx - 2R\right) \right\}_{t\in [0,T]} \Rightarrow \left\{ 2^H \int_0^t  (t-s) \eta(s)  dB_s\right\}_{t\in [0,T]} \,.
$$
\end{itemize}
 Here  $B$ is a standard Brownian motion and the above weak convergence takes place in the space of continuous functions $C([0,T])$.
\end{theorem}

Theorem  \ref{thm:TV-distance} is proved using a combination of Stein's method for normal approximation and Malliavin calculus, following the ideas introduced by Nourdin and Peccati in \cite{NP}. The main idea is as follows.  The total variation distance $d_{\rm TV}\left( F_R(t),Z\right) $ is bounded by $2 \sqrt{  {\rm Var} \langle DF_R(t), v_R \rangle_\HH}$, where $D$ is the derivative in the sense of Malliavin calculus,  $\HH$ is the Hilbert space
associated to the noise $W$ and $v_R$ is an $\HH$-valued random variable such that $F_R(t) =\delta(v_R)$, $\delta$ being the adjoint of the derivative operator, called the divergence or the Skorohod integral. A key new ingredient in the application of this approach is to use the representation of $F_R(t)$ as a stochastic integral   of $v_R$, taking into account that the It\^o-Walsh integral is a particular case of the Skorohod integral.

A similar problem for the stochastic heat equation on $\R$  has been recently  considered in  \cite{HNV18}, but only in the case of a space-time white noise. In this case, it was proved in \cite{HNV18} that  the limiting process in the functional central limit theorem is a martingale, which is not true for our wave equation.  Moreover, in the colored case $H\in (1/2,1)$  considered here, we have found the  surprising result that  the square moment  $ \E[ \sigma^2(u(s,y)) ]$  in the white noise case is replaced by the square of the first moment   $(\E[  \sigma(u(s,y)) ])^2$. Furthermore, the rate of convergence depends on the Hurst parameter $H$.

When $\sigma(u)=u$, the solution has an explicit Wiener chaos expansion. A natural question in this case is {\it whether the central limit is chaotic}, meaning that the projection on each Wiener chaos contributes to the limit. Such a phenomenon has been observed in other cases (see, for instance, \cite{HN}). We will show that for $H>1/2$ only the first chaos contributes to the limit, where as for $H=1/2$, we will see in  Remark  \ref{rem:nonchaotic} that the first chaos is not the only contributor in the limit and to check whether or not this central limit is chaotic, one shall go through the usual arguments for chaotic central limit theorem (see \cite[Section 8.4]{Eulalia}).

\medskip

The rest of the paper is organized as follows. In Section \ref{sec:prel} we recall some preliminaries on   Malliavin calculus and Stein's method. Sections \ref{sec:thm1} and \ref{sec:thm2} are devoted to the proofs of our main theorems.  We put the proof of a technical lemma  (Lemma \ref{lemma: iteration}) in the appendix. This lemma, which has an independent interest,  states that the $p$-norm of the Malliavin derivative $D_{s,y} u(t,x)$ can be estimated, up to constant that depends on $p$ and $t$, by the fundamental solution of the  wave equation $\frac 12\mathbf{1}_{\{ |x-y| \le t-s\}}$. 

  Along the paper we will denote by $C$ a generic constant that might depend on the fixed time $t$, the Hurst parameter $H$ and the non-linear coefficient $\sigma$,  and it can vary from line to line.
  
\section{Preliminaries}\label{sec:prel}

 We denote by $W= \{ W(t,x), t\ge 0, x\in \R\}$ a centered Gaussian family of random variables  defined in some probability space $(\Omega, \mathcal{F}, P)$,  with covariance function given by 
\[
\E\left[W(t,x) W(s,y)\right] =   \frac{s\wedge  t }{  2} \big( |x|^{2H}  + |y|^{2H} - |x-y| ^{2H}\big),
\]
where $ H \in [1/2, 1)$.

Let  $\HH_0$ be the Hilbert space defined as the completion of the set of step functions on $\R$ equipped with the inner product
\begin{align} \label{corfctH}
\langle \varphi, \phi \rangle_{\HH_0} 
= \begin{cases}
  H(2H-1) {\displaystyle \int_{\R^2}    \varphi(x) \phi(y) | x-y| ^{2H-2} \,dxdy }\qquad\text{if $H\in (1/2 ,1)$,} \\ 
 {\displaystyle  \int_{\R}    \varphi(x) \phi(x) \, dx}, \quad\qquad\qquad \qquad\qquad\qquad\qquad \text{if $H=1/2$.}
   \end{cases}
\end{align}
 Set $\HH= L^2( \R_+ ; \HH_0)$ and  notice that 
\[
\E\left[W(t,x) W(s,y)\right]  = \langle  \mathbf{1}_{[0,t]\times [0,x]} , \mathbf{1}_{[0,s]\times [0,y]} \rangle_{\HH},
\]
where, by  convention, $[0,x] = [-|x|, 0]$ if $x$ is negative. Therefore, the mapping $(t,x) \rightarrow W(t,x)$ can be extended to a linear isometry between $\HH$ and
the Gaussian subspace of $L^2(\Omega)$ generated by $W$. We denote this isometry by $\varphi \longmapsto W(\varphi)$.  

When $H=  1/2$,  the space  $\HH$ is simply $L^2( \R_+ \times \R)$ and $W(\varphi)$ is the Wiener-It\^o integral of  $\varphi$:
\[
W(\varphi)= \int_{\R_+ \times \R} \varphi(t,x) W(dt,dx).
\]
For $H \in (  1/2, 1)$, the space $L^{1/H} (\R)$  is known to be  continuously embedded into   $\HH_0$; see \cite{MMV01, PipirasTaqqu}.

For any $t\ge 0$, we denote  by  ${\mathcal F}_t$  the $\sigma$-field generated by  the random variables
$\{ W(s,x) :0\le s\le
t, x\in \R \}$. 
Then, for any adapted  $\HH_0$-valued   stochastic process $\{X(t),\;  t\ge 0 \}$
such that
\begin{equation}  \label{inte}
\int_0^\infty   \E[ \| X(t) \|_{\HH_0} ^2] d t <\infty,
\end{equation}
the following   stochastic integral
\begin{equation} \label{DW}
\int_0^\infty \int_{\R} X(s,y) W(d s, d y)
\end{equation}
is well-defined and satisfies the isometry property
\[
\E \left[  \left( \int_0^\infty \int_{\R} X(s,y) W(d s, d y) \right)^2 \right] = \E \left  ( \int_0^\infty    \| X(t) \|_{\HH_0} ^2 d t \right)\, .
\]

We will make use of the following lemma and the notation $\alpha_H = H(2H-1)$.

\begin{lemma} \label{lem1}
For any $H\in (1/2, 1)$, $s,t \ge 0$ and $x,\xi \in \R$, we have
       \begin{align}  \notag
&2\alpha_H  \int_{\R^2}    \mathbf{1}_{\{|x-y| \le  t\}} \mathbf{1}_{\{| \xi-z| \le s \}}  | y-z| ^{2H-2}  dydz\\ \notag
&\qquad =  \big\vert   x-\xi -t- s \big\vert ^{2H} +  \big\vert  x-\xi+t+s \big\vert ^{2H} \\  \label{ecu2}
&\quad \quad\quad\quad\quad -  \big\vert x-\xi+  t - s  \big\vert ^{2H} -   \big\vert x-\xi - t+s  \big\vert ^{2H}.
 \end{align}
\end{lemma}
\begin{proof}
Let $B^H$ be a two-sided  fractional Brownian motion with Hurst parameter $H$. That is, $B^H=\{ B^H_t, t \in \R\}$ is a centered Gaussian process with covariance
\[
\E[B^H_t B^H_s] = \frac 12 \left( |t|^{2H} + |s|^{ 2H} - |t-s| ^{2H} \right)\,, \,\,  s,t\in\RR \,.
\]
Notice that both sides of \eqref{ecu2} are equal to
$
2 \E\big[ (B^H_{x+t} - B^H_{x-t})( B^H_{\xi+s} - B^H_{\xi-s}) \big]
$, in view of \eqref{corfctH} and the above covariance structure.
  So the desired equality follows immediately. 
\end{proof}

The proof of our main theorems relies  on a combination of  Malliavin calculus and Stein's method. We will   introduce these tools in the next two subsections.  

 \subsection{Malliavin calculus} 

Now we recall some basic facts on Malliavin calculus associated with $W$. For a
detailed account of the Malliavin calculus with respect to a Gaussian process, we refer to Nualart \cite{Nualart}. 

 Denote by $C_p^{\infty}(\RR^n)$ the space of smooth functions with all their partial derivatives having at most polynomial growth at infinity. Let $\mathcal{S}$ be the space of simple functionals of the form 
$$
F = f(W(h_1), \dots, W(h_n))
$$ for $f\in C_p^{\infty}(\RR^n)$ and $h_i \in \HH$, $1\leq i \leq n$. Then, $DF$ is the $\HH$-valued random variable defined by
\begin{align*}
DF=\sum_{i=1}^n  \frac {\partial f} {\partial x_i} (W(h_1), \dots, W(h_n)) h_i\,.
\end{align*}
 The derivative operator $D$  is   closable   from $L^p(\Omega)$ into $L^p(\Omega;  \HH)$ for any $p \geq1$ and   we let $\mathbb{D}^{1,p}$ be the completion of $\mathcal{S}$ with respect to the norm
$$
\|F\|_{1,p} = \left(\E\big[ |F|^p \big] +   \E\big[  \|D F\|^p_\HH \big]   \right)^{1/p} \,.
$$
We denote by $\delta$ the adjoint of  $D$ given by the duality formula
$$
\E(\delta(u) F) = \E( \langle u, DF \rangle_\HH)
$$
for any $F \in \mathbb{D}^{1,2}$ and $u\in{\rm Dom} \, \delta \subset L^2(\Omega; \HH)$,  the domain of $\delta$. The operator $\delta$ is
also called the Skorohod integral,   because in the case of the Brownian motion, it coincides
with an extension of the It\^o integral introduced by Skorohod (see \cite{GT, NuPa}). 
More generally, in the context of our  Gaussian   noise $W$, any   adapted random field $X$ that satisfies (\ref{inte}) belongs to the domain of $\delta$ and $\delta (X)$ coincides with the Dalang-Walsh-type stochastic integral (\ref{DW}):
\[
\delta (X) = 
\int_0^\infty \int_{\R} X(s,y) W(d s, d y).
\]
As a consequence,  the mild formulation equation   \eqref{eq: mild} can  also be written as 
\begin{equation}\label{eq: mild Skorohod}
u(t,x) = 1 + \frac{1}{2} \delta\Big(  \mathbf{1}_{ \{\vert x- \ast \vert \le t -\cdot\} } \sigma\big(  u(\cdot, \ast) \big)\Big).
\end{equation}

It is known   that for any  $(t,x)\in\RR_+\times \RR$, the solution $u(t,x)$ to equation  (\ref{eq:heat-equation}) belongs to $\mathbb{D}^{1,p}$ for any $p\ge 2$ and the derivative satisfies the following linear  stochastic integral differential equation for $t\ge s$,
\begin{align}  \notag
D_{s,y}u(t,x) &=  \frac 12 \mathbf{1}_{\{ |x-y| \le t-s\}}   \sigma(u(s,y)) \\  \label{ecu1}
& \qquad +  \frac 12 \int_s^t \int_{\R}  \mathbf{1}_{\{ |x-z| \le t-r \}}    \Sigma(r,z) D_{s,y} u(r,z) W(dr,dz),
\end{align}
where $\Sigma(r,z)$ is an adapted process, bounded by the Lipschitz constant of $\sigma$
(we refer to the appendix for more details on the properties of the derivative). If $\sigma $ is continuously differentiable, then
$\Sigma(r,z)= \sigma'(u(r,z))$. This result is proved in    \cite[Proposition 2.4.4]{Nualart}  in the case of  the  stochastic heat equation with  Dirichlet boundary conditions on $[0,1]$ driven by a space-time white noise. Its proof can be easily extended to the wave  equation on $\R$ driven by the colored noise $W$. We also refer to   \cite{CHN18,NQ}  for additional references, where this result is used  for  $\sigma\in C^1(\RR)$.

In the end of this subsection, we record a technical result that is   essential for our arguments, and we postpone its proof to the Appendix. 

 \begin{lemma}\label{lemma: iteration}
  For any $p\in[ 2,+\infty)$,  $0  \leq  t  \le T $ and $x\in \R$, we have for almost every $(s,y)\in[0,T]\times \R$,
 \begin{equation}\label{ecu3}
  \| D_{s,y} u(t,x) \|_p   \le   C  \mathbf{1}_{\{ |x-y| \le t-s\}}
 \end{equation}
  for some constant $C = C_{T,p,H,\sigma}$ that depends on $T,p,H$ and the  function $\sigma$. 
  \end{lemma}

\subsection{Stein's method}
Stein's method is a probabilistic technique that allows one to measure the distance between a probability distribution and a target  distribution, notably the normal distribution. Recall that the total variation distance between two real  random variables $F$ and $G$  is defined by 
\begin{equation}
d_{\rm TV}(F,G) := \sup_{B \in \mathcal{B}(\RR)} \big\vert  P(F \in B) - P(G \in B) \big\vert\,, \label{TVDIST}
\end{equation}
where $\mathcal{B}(\RR)$ is the collection of all Borel sets in $\RR$.

\medskip

 The following theorem provides the well-known Stein's bound in  the total variation distance; see \cite[Chapter 3]{NP}.
 
\begin{theorem}\label{thm:Stein}
For $Z \sim \mathcal{N}(0,1)$ and for any  integrable random variable $F$, 
\begin{equation}
d_{\rm TV}(F, Z) \leq \sup_{f \in \mathscr{F}_{\rm TV} } \big\vert \E[ f'(F)] - \E [F f(F)]\big\vert \,,
\end{equation}
where   $\mathscr{F}_{\rm TV}$ is the class of continuously differentiable functions $f:\RR\to\RR$ such that
 $\|f\|_{\infty} \leq \sqrt{\pi/2} $ and $ \|f'\|_{\infty} \leq 2 $.
\end{theorem}
For a proof of this theorem, see \cite[Theorem 3.3.1]{NP}. 
Theorem  \ref{thm:Stein} can be combined with Malliavin calculus to get a very useful estimate (see \cite{HNV18, Eulalia,Zhou}).

\begin{proposition}\label{lem: dist}
Let $F=\delta (v)$ for some $\HH$-valued random variable $v\in{\rm Dom }\delta$. Assume  $F\in \mathbb{D}^{1,2}$ and $\E [F^2] = 1$ and  let $Z \sim \mathcal{N}(0,1)$.  Then we have 
\begin{equation}
d_{\rm TV}(F, Z) \leq 2 \sqrt{{\rm Var}\big[  \langle DF, v\rangle_{\HH} \big] }\,.
\end{equation}
\end{proposition}
 In the course of proving Theorem \ref{thm:functional-CLT},  we also need the following lemma, which is a generalization of    \cite[Theorem 6.1.2]{NP}; see \cite[Proposition 2.3]{HNV18}.

\begin{lemma}\label{lemma: NP 6.1.2}
Let $F=( F^{(1)}, \dots, F^{(m)})$ be a random vector such that $F^{(i)} = \delta (v^{(i)})$ for $v^{(i)} \in {\rm Dom}\, \delta$ and
$  F^{(i)} \in \mathbb{D}^{1,2}$, $i = 1,\dots, m$. Let $Z$ be an $m$-dimensional centered Gaussian  vector with covariance $(C_{i,j})_{ 1\leq i,j\leq m} $. For any  $C^2$ function $h: \R^m \rightarrow \R$ with bounded second partial derivatives, we have
\[
\big| \E [ h(F)] -\E [ h(Z)]  \big| \le \frac{m}{2}  \|h ''\|_\infty \sqrt{   \sum_{i,j=1}^m   \E \Big[ \big(C_{i,j} - \langle DF^{(i)}, v^{(j)} \rangle_{\HH}\big)^2 \Big] } \,,
\]
where $\|h ''\|_\infty : = \sup\big\{   \big\vert \frac{\partial^2}{\partial x_i\partial x_j} h(x) \big\vert\,:\, x\in\RR^m\, , \, i,j=1, \ldots, m \big\}$.
\end{lemma}
 
\section{Proof of Theorem \ref{thm:TV-distance}}
\label{sec:thm1}

We begin with the asymptotic variance of $F_R(t)$, as $R$ tends to infinity.  
 We need some preliminary results and notation. We fix $t > 0$ and  define
 \[
 \varphi_R(s,y) =\frac 12 \int_{-R} ^R  \mathbf{1}_{\{|x-y| \le t-s\}} dx.
 \]
 Notice that  $2\varphi_R(s,y)$ is the length of $[-R,R] \cap [ y-t+s, y+ t-s] $, so 
 \[
  \varphi_R(s,y) = \frac{1}{2} \Big( \big[ R\wedge (y+ t-s)\big] - \big[ (-R)   \vee (y-t+s)\big] \Big)_+.
\]
As a consequence, we deduce that  
\begin{center}
 $ \varphi_R(s,y) =0$,  if $|y| \geq R+t-s$;  and  $ \varphi_R(s,y) \le R \wedge (t-s)$.
  \end{center}
Set 
${\displaystyle
G_R =G_R(t): = \int_{-R}^Ru(t,x)dx - 2R.
}$
With this notation, we can write
\[
G_R=  \int_0^t \int_{\R} \varphi_R(s,y) \sigma (u(s,y)) W(ds,dy).
\]

The next lemma provides a useful formula.
\begin{lemma}\label{rem1}   Let $0 < a \leq b$ and define $\varphi_{a,R}(y ) = \frac{1}{2} \int_{-R}^R \mathbf{1}_{\{ \vert x - y \vert \leq a \}} \, dx$, then we have,  for any $R\ge 2b$,
\[
\int_{\R}  \frac{1}{R} \varphi_{a,R}(y) \varphi_{b,R}(y)  dy=   2ab - R^{-1} \Big( \frac{1}{2}ab^2 + \frac{1}{6}a^3 \Big)  .
\]
Therefore, $\lim_{R\to +\infty} \int_{\R}  \frac{1}{R} \varphi_{a,R}(y) \varphi_{b,R}(y)  dy =2ab$.
\end{lemma}
\begin{proof}
We can write
 \begin{align*}
&\quad \int_{\R}  \frac{1}{R} \varphi_{a,R}(y) \varphi_{b,R}(y)  \, dy  = \frac{1}{4R} \int_\R \int_{[-R, R]^2}    \mathbf{1}_{\{|\tilde{x}-y| \le a\}}    \mathbf{1}_{\{|x-y| \le b\}} \,  d\tilde{x} dx dy \\
 &= \frac{1}{4R}  \int_{[-R, R]^2}d\tilde{x} dx  \Big(  \mathbf{1}_{\{|\tilde{x}-x| \le b-a\}}  + \mathbf{1}_{\{ b-a < |\tilde{x}-x| \leq b+a\}}    \Big)   \int_{\RR }   \mathbf{1}_{\{|\tilde{x}-y| \le a ,|x-y| \le b\}} \,   dy   \\
 &=    \frac{1}{4R}  \int_{[-R, R]^2} \Big\{  \mathbf{1}_{\{|\tilde{x}-x| \le b-a\}}  (2a) + \mathbf{1}_{\{b-a <|\tilde{x}-x| \leq b+a\}} \big( a+b - \vert x- \tilde{x}  \vert \big) \Big\}      dx d\tilde{x}  \,,
 \end{align*}
which is equal to
$
2ab - R^{-1} \Big( \frac{1}{2}ab^2 + \frac{1}{6}a^3 \Big)$ for any $R\geq 2b$, as one can   verify. 
\end{proof}

The next result provides the asymptotic variance of $G_R(t)$ for $H=1/2$.
\begin{proposition}
\label{pro:covariance1}
Suppose $H= 1/2$. Denote $\xi(s)=   \E \big[\sigma^2(u(s,x))   \big] $, which does not depend on $x$ as a consequence of stationarity.
Then
$$
\lim_{R\to\infty}  \frac{1}{R}  \E\big[ G_R ^2\big] =  2 \int_0^t (t-s)^2\xi(s) ds.
$$
and $\E\big[ G_R^2\big] \geq {\displaystyle \left( \frac{5}{3} \int_0^t (t-s)^2\xi(s) ds\right)R}$ for any $R\geq 2t$.  
\end{proposition}
\begin{proof}
Thanks to the It\^o isometry,   we have 
\[
\E  [ G_R^2] =   \int_0^t \int_{\R} \varphi ^2 _R(s,y)     \E\big[ \sigma^2 (u(s,y))\big] dyds =
   \int_0^t  \xi(s) \int_{\R} \varphi ^2 _R(s,y)    \, dy\,ds.
\]
If  $R \geq 2t$,   we can see from  Lemma  \ref{rem1}     that
\begin{align}\label{use1}
 \frac{1}{R} \int_{\R} \varphi ^2 _R(s,y)     dy =  2(t-s)^2 \Big( 1- \frac{t-s}{3R} \Big) \in \Big[   \,\, \frac{5}{3}(t-s)^2 ,   2(t-s)^2  \Big] \,.
\end{align}
 This leads easily  to the results. \qedhere
 \end{proof}

Surprisingly, in the case $H>1/2$, we obtain a different formula for the asymptotic variance of $G_R$.

\begin{proposition}  \label{pro:covariance2}
Suppose $H\in (1/2,1)$.
Denote $\eta(s)=   \E [\sigma(u(s,x))] $, which does not depend on $x$  as a consequence of stationarity.
Then
$$
\lim_{R\to\infty} R^{-2H} \E[ G_R ^2] =  2^{2H} \int_0^t (t-s)^2 \eta^2(s)\, ds.
$$
\end{proposition}
\begin{proof}
Thanks to the It\^o isometry,  we have 
\[
\E  [ G_R^2] =   \alpha_H \int_0^t \int_{\R^2} \varphi_R(s,y)  \varphi_R(s,z)  \E\big[ \sigma (u(s,y))\sigma (u(s,z)) \big]   | y-z| ^{2H-2} dydzds,
\]
where $\alpha_H = H(2H-1)$. Keeping in mind   that $ \big\{\sigma\big(u(t,x)\big), x\in\RR\big\}$ is stationary, we  write
$
 \E\big[ \sigma (u(s,y))\sigma (u(s,z)) \big] =: \Psi (s, y-z).
 $
 Then,
\[
\E  [ G_R^2] =   \alpha_H \int_0^t \int_{\R^2} \varphi_R(s,\xi+z  )\varphi_R(s,z)    \Psi(s,\xi)  | \xi| ^{2H-2} d\xi dzds.
\]
We claim that
 \begin{equation} \label{ecu11}
 \lim_{|\xi|\to +\infty}  \sup_{0\le s\le t } | \Psi(s,\xi)  -\eta^2(s)|=0.
 \end{equation}
 In order to show (\ref{ecu11}), we apply a two-parameter version    of the Clark-Ocone formula (see \emph{e.g.} \cite[Proposition 6.3]{CKNP19}).    We can write
\[
  \sigma (u(s,y)) = \E[ \sigma (u(s,y))]
 + \int_0^s    \int_{\R}  \E\Big[  D_{r,\gamma}\big( \sigma (u(s,y)) \big) | \mathcal{F}_r\Big]  \, W(dr, d\gamma)
\]
and
\[
  \sigma (u(s,z)) = \E[ \sigma (u(s,z))]
 + \int_0^s    \int_{\R}  \E\Big[  D_{r,\beta}\big( \sigma (u(s,z))\big) | \mathcal{F}_r\Big] \, W(dr, d\beta).
\]
As a consequence,
\begin{equation} \label{ecu4}
 \E\Big[ \sigma \big(u(s,y)\big)\sigma \big(u(s,z)\big) \Big]  = \eta^2(s)+ T(s, y,z),
 \end{equation}
 where
\begin{align}
T(s, y,z)  \notag
&=\int_0^s    \int_{\R^2}   \E\Big\{  \E\Big[  D_{r,\gamma}\big( \sigma (u(s,y)) \big) | \mathcal{F}_r\Big]  \E\Big[ D_{r,\beta}\big( \sigma (u(s,z))\big) | \mathcal{F}_r\Big]  \Big\} \\  \label{fb2}
&\qquad\qquad\qquad\qquad\times  | \gamma-\beta|^{2H-2} d\gamma d\beta dr.
\end{align}
By the chain-rule for the derivative operator (see \cite[Proposition 1.2.4]{Nualart}),  
$$
 D_{r,\gamma}\big( \sigma (u(s,y)) \big) = \Sigma(s,y) D_{r,\gamma} u(s,y)$$
and 
$$ D_{r,\beta}\big( \sigma (u(s,z)) \big)= \Sigma(s,z) D_{r,\beta} u(s,z)
 $$
 with $\Sigma(s,y)$  an adapted random field  uniformly bounded by   the Lipschitz constant of $\sigma$, denoted by $L$. This implies,  
 using  (\ref{ecu3}),
 \begin{align} \notag
 &\quad  \Big\vert \E\Big\{  \E\Big[  D_{r,\gamma}\big( \sigma (u(s,y)) \big) | \mathcal{F}_r\Big]  \E\Big[ D_{r,\beta}\big( \sigma (u(s,z))\big) | \mathcal{F}_r\Big]  \Big\}  \Big\vert \\  \label{fb1}
 &  \leq   L^2 \|D_{r,\gamma} u(s,y)\|_2   \|D_{r,\beta} u(s,z)\|_2 \leq C  \mathbf{1}_{\{ |\gamma-y| \le s-r\}}\mathbf{1}_{\{ |\beta-z| \le s-r\}},
 \end{align}
 for some constant $C$. Therefore,  substituting (\ref{fb1}) into  (\ref{fb2}), we can write
  \begin{align}
   \vert T(s,y,z)  \vert     \leq   C  \int_0^s    \int_{\R^2}  \mathbf{1}_{\{ |\gamma-y| \le s-r\}}\mathbf{1}_{\{ |\beta-z| \le s-r\}} 
  | \gamma-\beta|^{2H-2} d\gamma d\beta dr. \label{exp1}
  \end{align}
If $ |y-z|  > 2s$, we have   
 $$
 \mathbf{1}_{\{ |\gamma-y| \le s-r\}}\mathbf{1}_{\{ |\beta-z| \le s-r\}}  | \gamma-\beta|^{2H-2} \leq  \mathbf{1}_{\{ |\gamma-y| \le s-r\}}\mathbf{1}_{\{ |\beta-z| \le s-r\}} \big( \vert y-z\vert -2s \big)^{2H-2}
 $$
 and therefore deduce from \eqref{exp1} that  (for $ |y-z|  > 2s$)
 \begin{align*}
   \vert T(s,y,z)  \vert  &   \leq   C  \int_0^s    \int_{\R^2} \mathbf{1}_{\{ |\gamma-y| \le s-r\}}\mathbf{1}_{\{ |\beta-z| \le s-r\}} \big( \vert y-z\vert -2s \big)^{2H-2} d\gamma d\beta dr \\
   &\leq  4Ct^3\big( \vert y-z\vert -2t \big)^{2H-2} \xrightarrow{\vert y-z\vert\to+\infty} 0 \,.
  \end{align*}
 Thus, claim \eqref{ecu11} is established  in view of formula (\ref{ecu4}).
 
  \medskip

Let us continue our proof of Proposition \ref{pro:covariance1}. We first show  that  the quantity
\begin{align}\label{quan}
 \frac{1}{R^{2H}}    \int_0^t \int_{\R^2}  \varphi_R(s,\xi+z  )\varphi_R(s,z)    \big[ \Psi(s,\xi) -  \eta^2(s)\big]    \vert \xi \vert ^{2H-2} d\xi dzds 
 \end{align}
converges to zero, as $R\to+\infty$.

By \eqref{ecu11}, we can find $K = K_\varepsilon > 0$ for any given $\varepsilon > 0$  such that  
   $$
\sup\Big\{  | \Psi(s, \xi) -\eta^2(s)| : s\in[0,t], | \xi| > K  \Big\} < \varepsilon.
   $$ 
    Now we divide the above integration domain into two parts $\vert \xi \vert\leq K$ and $\vert \xi \vert > K$.

\medskip
\noindent
{\it  Case (i):}    On the region $| \xi | \le K$,
by Cauchy-Schwarz inequality and \eqref{use1}, we get for $R\ge 2t$
\begin{align*}
\int_\RR  \varphi_R(s,\xi+z  )\varphi_R(x,z) \, dz& \leq \left( \int_\RR  \varphi_R^2(s,\xi+z  )\, dz \right)^{1/2} \left( \int_\RR  \varphi^2_R(s,z  )\, dz \right)^{1/2} \\
&= \int_\RR  \varphi^2_R(s,z  )\, dz =  2R(t-s)^2 \Big( 1- \frac{t-s}{3R} \Big) \,.
\end{align*}
Since   $\Psi(s,y-z) - \eta^2(s) = T(s,y,z )$ is uniformly bounded for $(s,y,z)\in[0,t]\times\R^2$,
\begin{align*}
&\quad R^{-2H}  \int_0^t \int_{\RR^2}  \varphi_R(s,\xi+z  )\varphi_R(s,z)  \big| \Psi(s,\xi) - \eta^2(s)  \big|  \mathbf{1}_{\{ \vert\xi\vert \leq K\}}     \vert \xi \vert ^{2H-2} \, d\xi dzds \\
&\leq C R^{-2H} \int_0^t \int_{-K}^{K} \left(    \int_\RR  \varphi_R(s,\xi+z  )\varphi_R(x,z) \, dz \right) \vert \xi \vert ^{2H-2} \, d\xi ds \\
&\leq  C R^{1-2H} \int_0^t (t-s)^2 \int_{-K}^{K} \vert \xi \vert ^{2H-2} \, d\xi ds \xrightarrow{R\to+\infty} 0 \,.
\end{align*}

\medskip
\noindent
{\it  Case (ii):}  On the region $| \xi | > K$, we know $| \Psi(s, \xi) -\eta^2(s)| < \varepsilon$ for $s\leq t$. Thus, 
    \begin{align*}
&  \quad R^{-2H}  \int_0^t \int_{\RR^2}  \varphi_R(s,\xi+z  )\varphi_R(s,z)  \big| \Psi(s,\xi) - \eta^2(s)  \big| \mathbf{1}_{\{ \vert\xi\vert > K\}}     \vert \xi \vert ^{2H-2} \, d\xi dzds\\
&\leq  \frac{\varepsilon}{R^{2H}}     \int_0^t \int_{\RR^2}  \varphi_R(s,\xi+z  )\varphi_R(s,z)     \vert \xi \vert ^{2H-2} \, d\xi dzds   \\
& = \frac{\varepsilon}{4R^{2H}}    \int_{-R}^R   \int_{-R}^R dx  dx'  \,\,    \int_0^t \int_{\RR^2} \mathbf{1}_{\{ \vert x-z\vert \leq t-s \}} \mathbf{1}_{\{ \vert x'-z'\vert \leq t-s \}}  \vert z-z'\vert^{2H-2}  ds dzdz'     \\
& \leq \frac{\varepsilon  t}{4R^{2H}}    \int_{-R}^R   \int_{-R}^R dx  dx'  \,\,     \int_{\RR^2} \mathbf{1}_{\{ \vert x-z\vert \leq t \}} \mathbf{1}_{\{ \vert x'-z'\vert \leq t \}}  \vert z-z'\vert^{2H-2}   dzdz'  \\
&=:  \varepsilon t^3 \mathcal{A}(R) \,.
\end{align*}
 We can rewrite $ \mathcal{A}(R)$, after  a change of variables and supposing $R>t$, in the following form
\begin{align*}
 \mathcal{A}(R) &= \int_{[-1,1]^2} \int_{\RR^2} \frac{R}{2t} \mathbf{1}_{\{ \vert x-z\vert \leq tR^{-1} \}}  \frac{R}{2t} \mathbf{1}_{\{ \vert x'-z'\vert \leq tR^{-1} \}}  \vert z-z'\vert^{2H-2}  \, dzdz' dxdx' \\
 &=  \int_{[-1,1]^2} \big(\zeta_R\ast g_2\big)(x, x')dx\, dx'  \xrightarrow{R\to+\infty} \| g_1\| _{L^1(\R^2)} \,, 
\end{align*}
where $g_m(z,z') :=  \vert z-z'\vert^{2H-2}  \mathbf{1}_{\{z\neq z'\}} \mathbf{1}_{\{z, z'\in[-m,m]\}}$, for $m=1,2$, are   integrable functions   on $\R^2$ and 
\[
\Big\{ \zeta_R(x,x') =   \frac{R}{2t} \mathbf{1}_{\{ \vert x\vert \leq tR^{-1} \}}  \frac{R}{2t} \mathbf{1}_{\{ \vert x'\vert \leq tR^{-1} \}} \,,\,\, R > t \Big\}
\]
defines an approximation of the identity.     This     leads to the asymptotic negligibility of the quantity \eqref{quan},  as $\varepsilon > 0$ is arbitrary.

Therefore, it suffices to show that
\begin{align}   \notag 
& \quad  \lim_{R\rightarrow \infty}  \frac{\alpha_H}{R^{2H}}    \int_0^t \eta^2(s)  \int_{\R^2}  \varphi_R(s,\xi+z  )\varphi_R(s,z)          \vert \xi \vert ^{2H-2} d\xi dzds  \\
 & = 2^{2H} \int_0^t(t-s)^2 \eta^2(s)ds. \label{quan2}
 \end{align}
The previous computations 
  imply that
\[
  \frac{1}{R^{2H}}\int_{\R^2}  \varphi_R(s,\xi+z  )\varphi_R(s,z)      \vert \xi \vert ^{2H-2} d\xi dz \leq 4t^2  \mathcal{A}(R)
\]
is uniformly bounded over $s\in[0,t]$. Moreover,  we can get  
\begin{align}
&\quad \lim_{R\to+\infty}  \frac{\alpha_H}{R^{2H}}\int_{\R^2}  \varphi_R(s,\xi  )\varphi_R(s,z)      \vert \xi -z \vert ^{2H-2} d\xi dz  \notag \\
&= (t-s)^2  \alpha_H \| g_1\|_{L^1(\R^2)} = 2^{2H}(t-s)^2 \,, \label{for:nonchaotic}
\end{align}
where the last equality follows from Lemma \ref{lem1}.   Hence  (\ref{quan2}) follows by  the dominated convergence theorem and this concludes our proof. \end{proof}

 It follows from the above two propositions that for fixed $t>0$,   the variance of $G_R(t)$, denoted by $\sigma_R^2$,   is $O(R^{2H})$.  
The next lemma  states that $R^{2H}$ is the exact order  under our standing assumption $\sigma(1) \not =0$, which is also a necessary condition to have this order.
Moreover,   $\sigma(1) \not =0$ is equivalent to $\sigma_R >0$ for all $R>0$. 
\begin{lemma} \label{lema1}
The following four conditions are equivalent:
\begin{itemize}
\item[(i)] $\sigma(1) =0$.
\item[(ii)]  $ \sigma_R =0$ for all $R>0$.
\item[(iii)]  $ \sigma_R =0$ for some $R>0$.
\item[(iv)] $\lim_{R\to\infty}\sigma_R^2 R^{-2H} =0$.
\end{itemize}
\end{lemma}
\begin{proof}  If $\sigma(1)=0$, then writing the solution as the limit of the Picard iterations starting with the constant solution  $1$, we obtain that $u(t,x)=1$ for all $(t,x)$. As a consequence,  $\sigma_R=0$ for all $R>0$ and (i) implies (ii).  Clearly (ii) implies (iii) and (iv). 
Now suppose that  (iv) holds. Then  Propositions \ref{pro:covariance1} and \ref{pro:covariance2} imply that   for almost every $s\in[0,t]$,
 \begin{itemize}
 \item[a)] $\E\big[ \sigma^2( u(s,y) ) \big] =0$ in the case $H=1/2$, 
\item[b)]  $\E\big[ \sigma( u(s,y) ) \big] =0$ in the case $H\in (1/2,1)$. 
\end{itemize}
By the $L^2(\Omega)$-continuity of the process $(s,y)\in\R_+\times\R \longmapsto u(s,y)$ (see \emph{e.g.} \cite[Theorem 13]{Dalang99}),  letting $s$ tend to $0$, we deduce that  $\sigma(1) = 0$ in both cases $H=1/2$ and $H\in(1/2, 1)$.

Finally, suppose that (iii) holds and assume that 
$H\in (1/2,1) $ (the proof in  the case $H= 1/2$ is similar).    By $L^2$-continuity, we can see that the function $\Psi(s,y):=  \E\big[ \sigma (u(s,0)) \sigma (u(s,y)) \big] $ is continuous on $\R_+\times\R$. 
Note that,  for  almost all $s \in [0,t]$, 
\begin{equation} \label{a11}
  \int_{\R^2} \varphi_R(s,y)\varphi_R(s,z)  \sigma (u(s,y)) \sigma (u(s,z))    |y-z| ^{2H-2} dydz  =0 
  \end{equation}
almost surely.  In the above integral, the variables $y$ and $z$ have support contained  in the interval $[-R-t, R+t]$. If $\sigma(1) \neq 0$, there exists a sufficiently small $\delta >0$ such that  for    $(s, y, z)\in[0,\delta] \times [ -R-t, R+t]^2$
      $$
    \E\big[  \sigma (u(s,y)) \sigma (u(s,z))   \big] =  \Psi(s, y-z)  \geq \vert \sigma(1)\vert^2/2  \,,
      $$ which is a contradiction to (\ref{a11}). Therefore, $\sigma(1) =0$ and (iii) implies 
  (i).
 \end{proof}

  \begin{remark}\label{rem:nonchaotic}
  It follows from Proposition \ref{pro:covariance2} that, if $H\in (1/2,1)$, the random variable $G_R$ is {\it not chaotic} in the linear case.  More precisely, when $\sigma(x)= x$,  the above proposition gives us 
  \[
{  \rm Var}\big( G_R) \sim R^{2H}   2^{2H} \int_0^t (t-s)^2 ds = \frac{ 4^Ht^3}{3} R^{2H} \quad \text{as $R\to+\infty$.}
  \]
 Due to linearity, one can obtain the Wiener-chaos expansion of $G_R$ easily:
  \[
  G_R =   \int_0^t\int_\R \varphi_R(s,y) \, W(ds, dy)  + \text{\it higher-order chaoses.}
  \]
 Then, the variance of the first chaos is equal to 
  \[
  \int_0^t \alpha_H \int_{\R^2}  \varphi_R(s,y) \varphi_R(s,z) \vert  y -z\vert^{2H-2} \, dydzds \sim  \frac{ 4^Ht^3}{3} R^{2H} \quad \text{as $R\to+\infty$,}
  \]
 which is a consequence of   \eqref{for:nonchaotic} and dominated convergence.  This shows that only the first chaos contribute to the limit, that is, there is a non-chaotic behavior of the spatial average of the \emph{linear} stochastic wave equation, when $H\in (1/2,1)$.
 
 For $H=1/2$ and $\sigma (x)=x$, we obtain from Proposition \ref{pro:covariance1}
   \[
{  \rm Var}\big( G_R) \sim 2R  \int_0^t (t-s)^2 \E[u^2(s,x)] ds \quad \text{as $R\to+\infty$,}
  \]
  whereas the variance of the  projection on the first chaos is, using Lemma  \ref{rem1},
  \[
    \int_0^t\int_\R \varphi_R^2(s,y) dyds  = \frac 23 Rt^3- \frac {t^4}{6}.
    \]
Notice that $ \E[u^2(s,x)] \ge ( \E [u(s,x) ])^2 =1$ and the inequality is strict for all $s\in (0,t]$  (otherwise $u(s,x)$ would be a constant). This  implies that   the first chaos  is not the only contributor to the limiting variance.
  \end{remark}

Before we give the proof of   Theorem \ref{thm:TV-distance}, by using the same argument as in the   proof of Propositions \ref{pro:covariance1} and \ref{pro:covariance2}, we obtain an asymptotic  formula for $\E\big[ G_R(t_i) G_R(t_j) \big]$ with $t_i, t_j\in\R_+$, which is a useful ingredient for our proof of functional central limit theorem. 

\begin{remark}\label{rem2} Suppose  $t_i , t_j\in\R_+$.  If $H=1/2$, we have
\begin{align*}
  \E\big[ G_R(t_i) G_R(t_j) \big] =   \int_0^{t_i\wedge t_j} \int_{\R} \varphi^{(i)}_R(s,y)  \varphi^{(j)}_R(s,y)  \xi(s)      dyds\,,
\end{align*}
where $ \varphi^{(i)}_R(s,y) =   \frac{1}{2} \int_{-R}^R  \mathbf{1}_{\{|x-y| \le t_i -s\}} \,  dx$ and we obtain
\begin{align*}
  \lim_{R\to+\infty}   \frac 1R \E\big[ G_R(t_i) G_R(t_j) \big]    =2  \int_0^{t_i\wedge t_j}  (t_i-s)(t_j-s) \xi(s) \, ds \,.
\end{align*}
In the case $H\in (1/2,1)$, we have $\E\big[ G_R(t_i) G_R(t_j) \big] $ equal to
\begin{align*}
    \alpha_H \int_0^{t_i\wedge t_j} \int_{\R^2} \varphi^{(i)}_R(s,y)  \varphi^{(j)}_R(s,z) \Psi(s,y-z)     | y-z| ^{2H-2} dydzds\,,
\end{align*}
 and we obtain
\begin{align*}
&\quad \lim_{R\to+\infty}  R^{-2H}  \E\big[ G_R(t_i) G_R(t_j) \big]   \\
&= \lim_{R\to+\infty} \alpha_H \int_0^{t_i\wedge t_j} ds~ \eta^2(s)   \int_{\R^2} \varphi^{(i)}_R(s,y)  \varphi^{(j)}_R(s,z)    | y-z| ^{2H-2} dydz \\
&= 2^{2H}   \int_0^{t_i\wedge t_j}  (t_i-s)(t_j-s) \eta^2(s) \, ds \,.
\end{align*}
\end{remark}

\bigskip

Now let us prove Theorem \ref{thm:TV-distance}.

\begin{proof}[Proof of Theorem \ref{thm:TV-distance}]

By Proposition  \ref{lem: dist},   if  $F = \delta(v) \in \mathbb{D}^{1,2}$ 
with $\E (F^2)=1$, we have
$$
d_{\rm TV}(F,Z)\leq 2\sqrt{{\rm Var}\big[\langle DF,v\rangle_{\HH}\big]}.
$$
Recall that in our case we have, as a consequence of  Fubini's theorem, that
\begin{equation*}
\begin{split}
F_R:=F_R(t) &= \frac{1}{\sigma_R} \int_{-R}^R \big[ u(t,x) - 1 \big]dx \\
&=  \frac{1}{\sigma_R}  \int_0^t \int_{\R} \varphi_R(s,y) \sigma (u(s,y)) W(ds,dy).
\end{split}
\end{equation*}
Similarly as in   \eqref{eq: mild Skorohod}, we can write, for any fixed $t > 0$,  
$
F_R =\delta(v_R)$   with $v_R(s,y) = \sigma_R^{-1}\mathbf{1} _{[0,t]}(s) \varphi_R(s,y) \sigma (u(s,y)).
$
Moreover,
$$
D_{s,y}F_R =\mathbf{1} _{[0,t]}(s) \frac{1}{\sigma_R}\int_{-R}^RD_{s,y}u(t,x)dx.
$$
Then, it follows from \eqref{ecu1} and Fubini's theorem that
 \begin{align*}  
&\quad \int_{-R}^RD_{s,y} u(t,x) \, dx  \\
&=  \varphi_R(s,y)  \sigma(u(s,y))     + \int_s^t \int_{\R}   \varphi_R(r,z)  \Sigma(r,z) D_{s,y} u(r,z) W(dr,dz).
\end{align*}
In what follows, we separate our proof into two cases:  $H=  1/2$ and $H>  1/2$.

\medskip
\noindent
{\it Case $H=  1/2$.} ~ We   write
 \[
\langle DF_R,v_R\rangle_{\HH}  := B_1 +B_2,
\]
where
\[
B_1=  \frac { 1} {\sigma_R^2} \int_0^t \int_{\R}  \varphi_R^2(s,y)   \sigma ^2(u(s,y)) \,ds\, dy
\]
and
\begin{align*}
 B_2  &=  \frac { 1 } {\sigma^2_R} \int_0^t \int_{\R} \varphi_R(s,y) \sigma(u(s,y)) \\
&\qquad \qquad \times \left( \int_s^t\int_\R \varphi_R(r,z) \Sigma(r,z) D_{s,y} u(r,z) W(dr,dz) \right)  \,    dyds.
\end{align*}
Notice that   for any process $\Phi= \{\Phi(s), s\in [0,t]\}$ such that  $ \sqrt{{\rm Var} (\Phi_s)} $ is integrable on $[0,t]$, it holds that
 \begin{align}
  \sqrt{ {\rm Var}  \left( \int_0^t \Phi_s ds \right)} \le \int_0^t  \sqrt{{\rm Var} (\Phi_s)} ds \,. \label{FACT00}
  \end{align}
So  we can write
 \[
 \sqrt{ {\rm Var} \big[ \langle DF_R,v_R\rangle_{\HH} \big] } \le  \sqrt{2}\Big(  \sqrt{ {\rm Var}(B_1) } +  \sqrt{ {\rm Var}(B_2) } \Big) \leq \sqrt{2}(A_1+A_2) \,,
 \]
 with  
\begin{align*}
  A_1 &=\frac{1}{\sigma^2_R} \int_0^t    \left(   \int_{\R^2}  \varphi^2_R(s,y)\varphi^2_R(s,y') 
{\rm Cov} \Big[ \sigma^2 \big(u(s,y)\big)  , \sigma ^2\big(u(s, y')\big)  \Big] dydy'
 \right)^{1/2} ds
\intertext{and}
  A_2 &=  \frac{1}  {\sigma_R^2} \int_0^t \Bigg(   \int_{\R^3}  \int_s^t   \varphi^2_R(r,z)  \varphi_R(s,y)   \varphi_R(s,y') \\
  &\times
     \E  \Big[    \Sigma^2(r,z) D_{s,y} u(r,z)  D_{s,y'} u(r,z) 
  \sigma\big(u(s,y)\big)\sigma\big(u(s,y')\big) \Big] \, dydy'    dz  dr
   \Bigg)^{  1/2} ds \,.
 \end{align*}
 Then the rest of the proof for this case ($H=1/2$) consists in estimating $A_2$ and $A_1$. The proof will be done in two steps.
 
 \medskip
 \noindent {\it  Step 1:} ~ Let us proceed with the estimation of $A_2$.  As before, denote by $L$ the  Lipschitz constant of $\sigma$ and  for $p\ge 2$, as a consequence of stationarity, we write 
 \begin{equation} \label{Kp}
 K_p(t) =\sup_{0\le s\le t} \sup_{y\in \R} \| \sigma(u(s,y)) \|_p = \sup_{0\le s\le t}   \| \sigma(u(s,0)) \|_p.
 \end{equation}
 Then, 
    \begin{align*}
&  \quad   \Big\vert  \E  \Big(    \Sigma^2(r,z) D_{s,y} u(r,z)  D_{s,y'} u(r,z) \sigma\big(u(s,y)\big)\sigma\big(u(s,y')\big) \Big)  \Big\vert \\
&  \le K_4^2(t) L^2 \| D_{s,y} u(r,z) \|_4 \| D_{s, y'} u(r,z) \|_4 \leq CK_4^2(t) L^2  \mathbf{1}_{\{|y-z| \le r-s\}}     \mathbf{1}_{\{|y'-z| \le r-s\}} \,,
 \end{align*}
 where the last inequality follows from  Lemma \ref{lemma: iteration}. This implies, together with Proposition \ref{pro:covariance1}, that, for any $R\ge 2t$, 
          \begin{align*}
  A_2 \leq&  \frac{C}  {R} \int_0^t \Bigg(   \int_{\R^3}  \int_s^t   \varphi_R^2(r,z)  \varphi_R(s,y)   \varphi_R(s,y') \\
  &\qquad    \qquad\times
     \mathbf{1}_{\{|y-z| \le r-s\}}\mathbf{1}_{\{|y' -z| \le r-s\}} dydy'    dz  dr
   \Bigg)^{  1/2} ds .
 \end{align*}
 Using first  $\varphi_R(s,y)   \varphi_R(s,y') \le (t-s)^2$ and then integrating in $y$ and $y'$, we obtain
\[
  A_2 \leq  \frac{C}  {R} \int_0^t \left(    \int_s^t  \int_{\R}  \varphi_R^2(r,z)  
    dz  dr
   \right)^{  1/2} ds \leq \frac{C}  {R} \int_0^t \left(    \int_s^t    2R(t-r)^2  dr
   \right)^{  1/2} ds \,,
\]
where the last inequality follows from \eqref{use1}. Therefore, we have $A_2
 \le   C/ \sqrt{R}$ for any $R\geq 2t$.

 \medskip
 \noindent {\it Step 2:}  ~ Consider now the term $A_1$.
We begin with a bound for the covariance 
 \[
   {\rm Cov} \Big[ \sigma^2 \big(u(s,y)\big)  , \sigma ^2\big(u(s, y')\big)  \Big] \,.
  \]
Using  a version of Clark-Ocone formula for two-parameter processes, we    write
\[
\sigma^2 (u(s,y))   =  \E\big[\sigma^2 (u(s,y))  \big] +  \int_0^s \int_{\R} \E\big[ D_{r,z} \big(\sigma^2(u(s,y))\big) | \mathcal{F}_r \big] \, W(dr,dz).
\]
  Then, ${\rm Cov} \Big[ \sigma^2 \big(u(s,y)\big)  , \sigma ^2\big(u(s, y')\big)  \Big]$ is equal to
       \begin{align*}
         \int_0^s \int_{\R} \E \Big\{ \E\big[ D_{r,z} \big(\sigma^2(u(s,y)) \big)  | \mathcal{F}_r \big] 
        \E\big[ D_{r,z} \big( \sigma^2(u(s,y'))\big)   | \mathcal{F}_r \big]  \Big\}
         dzdr .
                 \end{align*}
                 By  the chain rule, we have
$
                 D_{r,z} \big(\sigma^2(u(s,y)) \big)  =  2 \sigma(u(s,y)) \Sigma(s,y) D_{r,z} u(s,y) 
$, thus             
                 $
              \left\|  \E[ D_{r,z} \big(\sigma^2(u(s,y)) \big) | \mathcal{F}_r ]    \right\| _2 \le 2K_4(t) L  \left\| D_{r,z} u(s,y)\right\|_4 .
              $
                    Then, using  Lemma \ref{lemma: iteration},  we can write
        \begin{align*}
        &   \left|  {\rm Cov} \Big[ \sigma^2 \big(u(s,y)\big)  , \sigma ^2\big(u(s, y')\big)  \Big]  \right|  \le  C   \int_0^s \int_{\R }  \left\| D_{r,z} u(s,y)\right\|_4  \left\| D_{r,z} u(s,y')\right\|_4   dz dr \\
        & \quad \le C  \int_0^s \int_{\R }  \mathbf{1}_{\{ |y-z| \le s-r\}}     \mathbf{1}_{\{ |y'-z| \le s-r\}}     dz dr  \le C  \mathbf{1}_{\{|y-y'| \le 2s\}} \,.
                 \end{align*}
      This leads to the following estimate for $A_1$, for any $R\geq 2t$:
              \[
  A_1 \le \frac{C}{R} \int_0^t    \left(   \int_{\R^2}  \varphi^2_R(s,y)\varphi^2_R(s,y') 
   \mathbf{1}_{\{|y-y'| \le 2s\}} dydy'
 \right)^{ 1/2} ds.
\]
Since $  \varphi^2_R(s,y)\varphi^2_R(s,y')  \le (t-s)^4 \mathbf{1}_{\{|y| \vee |y'| \le R+t-s\}}$, we   get $A_1 \leq   C /\sqrt{R}$ for $R\geq 2t$. This concludes our proof for the case $H=1/2$.

\bigskip
 The proof for the other case is more involved but we can proceed in similar steps.

\medskip
\noindent
{\it Case $H >  1/2$}. ~  In this case, we   write
$
\langle DF_R,v_R\rangle_{\HH}  := B_1 +B_2,
$
where
\begin{align*}
B_1 &=  \frac { \alpha_H } {\sigma_R^2} \int_0^t \int_{\R^2}  \varphi_R(s,y)\varphi_R(s, y')   \sigma \big(u(s,y)\big)  \sigma\big(u(s,y')\big)    |y-y'| ^{2H-2} \, dydy'ds
\intertext{and}
B_2 &=   \frac { \alpha_H } {\sigma_R^2} \int_0^t \int_{\R^2} 
\left( \int_s^t \varphi_R(r,z) \Sigma(r,z) D_{s,y} u(r,z) \, W(dr,dz) \right) \\
&\qquad  \qquad\qquad\qquad\qquad  \times \varphi_R(s,y') \sigma\big(u(s,y')\big) | y-y'| ^{2H-2}~ dydy' ds.
\end{align*}
This decomposition implies
 $ \sqrt{ {\rm Var} \big[\langle DF_R,v_R\rangle_{\HH} \big] } \le  \sqrt{2}(A_1+A_2),
 $
 with  
 \begin{align*}
  A_1&=\frac{\alpha_H}{\sigma^2_R} \int_0^t    \Bigg(   \int_{\R^4}  \varphi_R(s,y)\varphi_R(s,y') 
\varphi_R(s, \tilde{y})\varphi_R(s, \tilde{y}')   |y-y'| ^{2H-2}  |\tilde{y}-\tilde{y}'| ^{2H-2}  \\
&\quad    \times {\rm Cov} \Big[ \sigma \big(u(s,y)\big)  \sigma\big(u(s,y')\big) , \sigma \big(u(s,\tilde{y})\big)  \sigma\big(u(s,\tilde{y}')\big) \Big] ~ dy dy' d\tilde{y}d\tilde{y}'
 \Bigg)^{1/2} ds
\intertext{and}
  A_2 &=  \frac{\alpha_H^{3/2}}  {\sigma_R^2} \int_0^t \Bigg(   \int_{\R^6}  \int_s^t   \varphi_R(r,z)\varphi_R(r,\tilde{z} )  \varphi_R(s,y')   \varphi_R(s,\tilde{y}') \\
  &\qquad\quad\times
     \E  \Big\{    \Sigma(r,z) D_{s,y} u(r,z)  \Sigma(r,\tilde{z}) D_{s,\tilde{y}} u(r,\tilde{z}) \sigma\big(u(s,\tilde{y}')\big)\sigma\big(u(s,y')\big) \Big\}  \\
& \qquad\qquad  \times   | y-y'| ^{2H-2} | \tilde{y}-\tilde{y}'| ^{2H-2}  |z-\tilde{z}| ^{2H-2} dydy'  d\tilde{y} d\tilde{y}'   dz d\tilde{z} dr
   \Bigg)^{1/2} ds \,.
 \end{align*}
 The proof will be done in two steps:
 
 \medskip
 \noindent {\it Step 1:} ~
 Let us first estimate the term $A_2$. Recall that  $L$ denotes the  Lipschitz constant of $\sigma$ and recall the notation $K_p(t)$ ($p\geq 2$) introduced in (\ref{Kp}). We can write 
    \begin{align*}
&\quad  \Big\vert    \E  \Big\{    \Sigma(r,z) D_{s,y} u(r,z)  \Sigma(r,\tilde{z}) D_{s,\tilde{y}} u(r,\tilde{z}) \sigma\big(u(s,\tilde{y}')\big)\sigma\big(u(s,y')\big) \Big\} \Big\vert  \\
&   \le K_4^2(t) L^2 \| D_{s,y} u(r,z) \|_4 \| D_{s, \tilde{y}} u(r,\tilde{z}) \|_4 \leq C ~ \mathbf{1}_{\{|y-z| \le r-s\}} \mathbf{1}_{\{|\tilde{y}-\tilde{z}| \le r-s\}}
 \end{align*}
where the last inequality follows from  Lemma \ref{lemma: iteration}.  

Now we  derive  from Proposition \ref{pro:covariance2} the following estimate: For fixed $t>0$, there exists a constant $R_t$ that depends on $t$ such that for any $R\geq R_t$,  
          \begin{align*}
  A_2 &\leq   \frac{C}  {R^{2H}} \int_0^t \Bigg(   \int_{\R^6}  \int_s^t   \varphi_R(r,z)\varphi_R(r,\tilde{z} )  \varphi_R(s,y')   \varphi_R(s,\tilde{y}') 
     \mathbf{1}_{\{|y-z| \le r-s, |\tilde{y}-\tilde{z}| \le r-s\}} \\
&\qquad\qquad \quad \times   | y-y'| ^{2H-2} | \tilde{y}-\tilde{y}'| ^{2H-2}  |z-\tilde{z}| ^{2H-2} dydy'  d\tilde{y} d\tilde{y}'   dz d\tilde{z} dr
   \Bigg)^{1/2} ds,
 \end{align*}
 where $C$ is a constant that depends on $t,p,H$ and $\sigma$.
 
 The  integral  in the spatial variable term can be rewritten as
        \begin{align*} 
  \mathbf{I}&:=   \frac{1}{16}   \int_{[-R,R]^4}    \int_{\R^6}      \mathbf{1}_{\{  \vert x-z\vert\leq t-r, \vert \tilde{x} -\tilde{z}\vert \leq t-r,   |x'-y'| \le t-s,   |\tilde{x}'-\tilde{y}'| \le t-s,  |y-z| \le r-s,   |\tilde{y}-\tilde{z}| \le r-s\}} \\
&\qquad\qquad  \times   | y-y'| ^{2H-2} | \tilde{y}-\tilde{y}'| ^{2H-2}  |z-\tilde{z}| ^{2H-2} ~ dx d\tilde{x} dx'   d\tilde{x}'dydy'  d\tilde{y} d\tilde{y}'   dz d\tilde{z} \\
&=    \frac{R^{6H+4}}{16}     \int_{[-1,1]^4}    \int_{\R^6}   \mathbf{1}_{\{  \vert x-z\vert\leq \frac{t-r}{R}, \vert \tilde{x} -\tilde{z}\vert \leq \frac{t-r}{R},   |x'-y'| \le \frac{t-s}{R},   |\tilde{x}'-\tilde{y}'| \le \frac{t-s}{R},  |y-z| \le \frac{r-s}{R},   |\tilde{y}-\tilde{z}| \le \frac{r-s}{R}\}} \\
&\qquad\qquad  \times   | y-y'| ^{2H-2} | \tilde{y}-\tilde{y}'| ^{2H-2}  |z-\tilde{z}| ^{2H-2} ~ dx d\tilde{x} dx'   d\tilde{x}'dydy'  d\tilde{y} d\tilde{y}'   dz d\tilde{z},
 \end{align*}
 where the second equality follows from a simple change of variables.  
  Assuming $R\ge t$  and integrating in the variables $x, x', \tilde{x}, \tilde{x}'\in[-1,1]$, we have 
\begin{align*}
   \mathbf{I} & \le R^{6H} \int_{\R^6}   \mathbf{1}_{\{      |y-z| \le \frac{t}{R},   |\tilde{y}-\tilde{z}| \le \frac{t}{R}\}}  \mathbf{1}_{[-2,2]} (z) \mathbf{1}_{[-2,2]} (\tilde{z})
   \mathbf{1}_{[-2,2]} (y') \mathbf{1}_{[-2,2]} (\tilde{y}')
   \\
&\qquad\qquad  \times   | y-y'| ^{2H-2} | \tilde{y}-\tilde{y}'| ^{2H-2}  |z-\tilde{z}| ^{2H-2} ~  dydy'  d\tilde{y} d\tilde{y}'   dz d\tilde{z}.
   \end{align*}
 If $K= \sup_{y\in [-3,3]} \int_{-2}^2 | y-y'| ^{2H-2} dy'$,  then for $R\geq R_t + t$,
\[
    \mathbf{I}  \le K^2  R^{6H} \int_{\R^4}   \mathbf{1}_{\{      |y-z| \le \frac{t}{R},   |\tilde{y}-\tilde{z}| \le \frac{t}{R}\}}  \mathbf{1}_{[-2,2]} (z) \mathbf{1}_{[-2,2]} (\tilde{z})
 |z-\tilde{z}| ^{2H-2} ~  dy   d\tilde{y}   dz d\tilde{z}.
 \]
 Finally, integrating in $y$ and $\tilde{y}$, yields for $R\geq R_t + t$,
 \[
    \mathbf{I}  \le 36 K^2  R^{6H-2} \int_{[-2,2]^2}     
 |z-\tilde{z}| ^{2H-2} ~      dz d\tilde{z}.
\]
  As a consequence, 
   \[
   A_2 \le C R^{H-1}
   \]
for $R$ big enough.

\medskip
 \noindent {\it Step 2:} ~ 
 It remains to  estimate the term $A_1$. We will show $A_1 \le C R^{H-1}$    for  $R$ big enough. We begin with a bound for the covariance 
 \[
 {\rm Cov} \Big[ \sigma \big(u(s,y)\big)  \sigma\big(u(s,y')\big) , \sigma \big(u(s,\tilde{y})\big)  \sigma\big(u(s,\tilde{y}')\big) \Big]  \,.
   \]
According to   a version of Clark-Ocone formula for two-parameter processes, we write
  \begin{align*}
\sigma \big(u(s,y)\big)  \sigma\big(u(s,y')\big) &=  \E\big[\sigma \big(u(s,y)\big)  \sigma\big(u(s,y')\big)\big] \\
&\quad+  \int_0^s \int_{\R} \E\Big[ D_{r,z} \Big(\sigma\big(u(s,y)\big)   \sigma\big(u(s,y')\big) \Big) | \mathcal{F}_r \Big] W(dr,dz).
  \end{align*}
  Then,
       \begin{align*}
        &\quad   {\rm Cov} \Big[ \sigma \big(u(s,y)\big)  \sigma\big(u(s,y')\big) , \sigma \big(u(s,\tilde{y})\big)  \sigma\big(u(s,\tilde{y}')\big) \Big]  \\
        & =  \alpha_H \int_0^s \int_{\R^2} \E \Big\{ \E\Big[ D_{r,z} \Big(\sigma\big(u(s,y)\big)   \sigma\big(u(s,y')\big) \Big)| \mathcal{F}_r \Big] \\
        &\qquad\qquad\qquad \times \E\Big[ D_{r,z'} \Big( \sigma\big(u(s,\tilde{y})\big)   \sigma\big(u(s,\tilde{y}')\big)\Big) | \mathcal{F}_r \Big]  \Big\}         |z-z'|^{2H-2}  ~dz dz'dr .
                 \end{align*}
                 Applying  the chain rule for Lipschitz functions (see \cite[Proposition 1.2.4]{Nualart}), we have
                        \begin{align*}
                 D_{r,z} \Big(\sigma\big(u(s,y)\big)   \sigma\big(u(s,y')\big)\Big)  &= \sigma\big(u(s,y)\big) \Sigma(s,y') D_{r,z} u(s,y')\\
              & \qquad  + \sigma\big(u(s,y')\big) \Sigma(s,y) D_{r,z} u(s,y) 
              \end{align*}
               and therefore, $\Big\|  \E\big[ D_{r,z} \big(\sigma (u(s,y) )   \sigma (u(s,y') ) \big)| \mathcal{F}_r \big]    \Big\| _2$ is bounded by 
                 \[
                2K_4(t) L \big\{   \| D_{r,z} u(s,y) \|_4+  \| D_{r,z} u(s,y') \|_4\big\}\, .
              \]
               Applying   Lemma \ref{lemma: iteration},  we get $  | {\rm Cov} \big[ \sigma  (u(s,y) )  \sigma (u(s,y') ) , \sigma  (u(s,\tilde{y}) )  \sigma (u(s,\tilde{y}') ) \big] |
                 $ bounded by 
        \begin{align*}
        &     4L^2 K_4^2(t)    \int_0^s \int_{\R^2 } (  \left\| D_{r,z} u(s,y)\right\|_4 +\left\| D_{r,z} u(s,y')\right\|_4  ) \\
        &\qquad\quad \qquad\qquad \times   \Big(   \| D_{r,z'} u(s,\tilde{y}) \|_4 + \| D_{r,z'} u(s,\tilde{y}') \|_4  \Big)  |z-z'| ^{2H-2} dz dz'dr \\
        &   \le C  \int_0^s \int_{\R^2 }  \Big( \mathbf{1}_{\{ |y-z| \le s-r\}}  + \mathbf{1}_{\{ |y'-z| \le s-r\}}   \Big) \\
        & \qquad\qquad\qquad\qquad \times \Big( \mathbf{1}_{\{ |\tilde{y}-z'| \le s-r\}}  + \mathbf{1}_{\{ |\tilde{y}'-z'| \le s-r\}}   \Big)   |z-z'| ^{2H-2} dz dz'dr . 
                 \end{align*}
So the spatial integral in the expression of $A_1$ can be bounded by 
\begin{align*}
&\quad  \mathbf{J}:=C\int_0^s\int_{\R^6}  \varphi_R(s,y)\varphi_R(s,y') 
\varphi_R(s, \tilde{y})\varphi_R(s, \tilde{y}')     \\
& \quad\quad\quad\times  |y-y'| ^{2H-2}  |\tilde{y}-\tilde{y}'| ^{2H-2}  |z-z'| ^{2H-2}    \Big( \mathbf{1}_{\{ |y-z| \le s-r\}}  + \mathbf{1}_{\{ |y'-z| \le s-r\}}   \Big)  \\
&\qquad \qquad \times\Big( \mathbf{1}_{\{ |\tilde{y}-z'| \le s-r\}}  + \mathbf{1}_{\{ |\tilde{y}'-z'| \le s-r\}}   \Big) dy dy' d\tilde{y}d\tilde{y}' dz dz'dr \\
&=  C\int_0^s  \int_{[-R, R]^4} \int_{\R^6}  \mathbf{1}_{\{   \vert x - y \vert \vee  \vert x' - y' \vert \vee  \vert \tilde{x} -  \tilde{y} \vert  \vee \vert \tilde{x}' -  \tilde{y}' \vert \leq t-s \}       }      \\
& \quad\quad\quad\times  |y-y'| ^{2H-2}  |\tilde{y}-\tilde{y}'| ^{2H-2}  |z-z'| ^{2H-2}     \Big( \mathbf{1}_{\{ |y-z| \le s-r\}}  + \mathbf{1}_{\{ |y'-z| \le s-r\}}   \Big)  \\
&\qquad \qquad \times\Big( \mathbf{1}_{\{ |\tilde{y}-z'| \le s-r\}}  + \mathbf{1}_{\{ |\tilde{y}'-z'| \le s-r\}}   \Big) dx dx' d\tilde{x} d\tilde{x}'   dy dy' d\tilde{y}d\tilde{y}' dz dz'dr \\
&\leq  4Ct  \int_{[-R, R]^4} \int_{\R^6}  \mathbf{1}_{\{   \vert x - y \vert \vee  \vert x' - y' \vert \vee  \vert \tilde{x} -  \tilde{y} \vert  \vee \vert \tilde{x}' -  \tilde{y}' \vert \leq t \}       }          \mathbf{1}_{\{ |y-z| \le t\}}    \mathbf{1}_{\{ |\tilde{y}-z'| \le t\}}     \\
& \quad\quad\qquad\times  |y-y'| ^{2H-2}  |\tilde{y}-\tilde{y}'| ^{2H-2}  |z-z'| ^{2H-2}   dx dx' d\tilde{x} d\tilde{x}'   dy dy' d\tilde{y}d\tilde{y}' dz dz' \,,
\end{align*}                 
          due to symmetry.  Then, it follows from the exactly the same argument as in the estimation of $\mathbf{I}$ in  the previous step that
                $
             \mathbf{J}  $ is bounded by $C R^{6H-2}$   for $R$ big enough. This gives us the desired estimate for $A_1$ and finishes the proof.  \qedhere

 \end{proof}
 
 \section{Proof of Theorem \ref{thm:functional-CLT}}
\label{sec:thm2}
We begin with the following result that ensures tightness. 

\begin{proposition}
\label{pro:tightness}
Let $u(t,x)$ be the solution to equation \eqref{eq:heat-equation}. Then for any $0\leq s < t\leq T$ and any $p\ge 2$, there exists a constant $C_{p,T}$,  depending on $T$ and $p$,  such that for any $R\geq T$,
\begin{align}\label{tendue}
\E \left(~ \left|\int_{-R}^R u(t,x)dx - \int_{-R}^Ru(s,x)dx\right|^p ~ \right) \leq C_{p,T}R^{pH}(t-s)^{p}~ .
\end{align}
\end{proposition}

\begin{proof}
Let us assume that $s < t$. We can write
\begin{align*}
2\int_{-R}^R \big[ u(t,x) -   u(s,x)\big]~ dx =
\int_0^T  \int_\R  (\varphi_{t,R}(r,y)  - \varphi_{s,R}(r,y) )\sigma(u(r,y))W(dr,dy),
\end{align*}
where $ 
\varphi_{t,R}(r,y) = \mathbf{1}_{\{r\le t\}}   \int_{-R} ^R   \mathbf{1} _{\{ |x-y| \le t-r\}} dx .
$
      The rest of our proof consists of two parts.

\medskip
\noindent
{\it Step 1:}   ~ Suppose that $H=  1/2$.  Using Burkholder-Davis-Gundy inequality and Minkowski's inequality, we get, for some absolute constant $c_p\in(0,+\infty)$,
\begin{align*}
&\quad \E\left(\left|\int_{-R}^R u(t,x)dx - \int_{-R}^Ru(s,x)dx\right|^p \right)\\
 & \leq   c_p \E\left[ \left( \int_0^T \int_{\R}      \Big(   \varphi_{t,R}(r,y)  - \varphi_{s,R}(r,y) \Big) ^2 \sigma^2\big(u(r,y)\big)  \, dy  dr\right)^{p/2} ~\right]\\
 &   \leq  c_p \left(\int_0^T\int_{\R}      \Big(   \varphi_{t,R}(r,y)  - \varphi_{s,R}(r,y) \Big)^2 \big\|\sigma(u(r,y)) \big\|_p^2  ~dy   dr\right)^{p/2}\\ 
 &    \leq c_p K^p_p(T) \left(\int_0^T\int_{\R}     \Big( \varphi_{t,R}(r,y)  - \varphi_{s,R}(r,y)  \Big)^2 dy
  dr\right)^{p/2}\,,
  \end{align*}
    where $K_p(T)$  has been defined in (\ref{Kp}).  Now we notice that
 \begin{align}
 &\quad \big\vert  \varphi_{t,R}(r,y)  - \varphi_{s,R}(r,y) \big\vert  \notag \\
 &\leq  \mathbf{1}_{\{r\le s\}} \int_{-R} ^R  \left|\mathbf{1}_{\{|x-y| \le t-r\}}
 - \mathbf{1}_{\{|x-y| \le s-r\}} \right| dx  + 
 \mathbf{1}_{\{s< r \le t \}}\int_{-R} ^R \mathbf{1}_{\{|x-y| \le t-r\}}dx  \notag \\
 &\le  2\Big( t-s  +  (t-r)\mathbf{1}_{\{s< r \le t \}} \Big) \mathbf{1}_{ \{ |y| \le R+t\}} \leq  4(t-s) \mathbf{1}_{ \{ |y| \le R+t\}} \,. \label{est11}
  \end{align}
 This implies for $R\geq T$,
\[
 \int_0^T\int_{\R}     \big( \varphi_{t,R}(r,y)  - \varphi_{s,R}(r,y)  \big)^2 dyds \le 64T R(t-s)^2 ,~\text{and thus establishes \eqref{tendue}.}
\]
   
\medskip
\noindent
{\it Step 2:}  ~ Suppose that $H\in (  1/2,1)$.
In the same way, we   write
\begin{align}  \notag  
&\quad  \E\left(\left|\int_{-R}^R u(t,x)dx - \int_{-R}^Ru(s,x)dx\right|^p \right)\\
 &  \leq   c_p \label{a2}
  \E \Bigg[  \Big( \int_0^T   \left\|  \big(\varphi_{t,R}(r,\cdot)  - \varphi_{s,R}(r,\cdot) \big)\sigma\big(u(r,\cdot)\big)  \right\|_{\HH_0}^2 dr\Bigg)^{p/2}\Bigg].
  \end{align}
 As mentioned in Section \ref{sec:prel}, for $H\in(1/2, 1)$, the space $L^{1/H}(\R)$  is  continuously embedded  into     $\mathfrak{H}_0$.   Consequently, there is a constant $C_H>0$, depending on $H$, such that
 \begin{align} \notag
 & \quad \left\|  \big(\varphi_{t,R}(r,\cdot)  - \varphi_{s,R}(r,\cdot) \big)\sigma\big(u(r,\cdot)\big)  \right\|_{\HH_0}^2 \\ \label{a1}
 & \le  C_H   \left( \int_{\R}  \big|\varphi_{t,R}(r,y)  - \varphi_{s,R}(r,y) \big| ^{1/H} |\sigma\big(u(r,y)\big) |^{1/H} dy \right)^{2H}.
 \end{align}
 Substituting (\ref{a1}) into (\ref{a2}) and applying H\"older's and Minkowski's inequalities, we can write
 \begin{align*} 
&\quad  \E\left(\left|\int_{-R}^R u(t,x)dx - \int_{-R}^Ru(s,x)dx\right|^p \right)\\
 &  \le  c_p C_H^{p/2}  T^{p/2-1} 
 \int_0^T  \E \left[ \left( \int_{\R}  \big|\varphi_{t,R}(r,y)  - \varphi_{s,R}(r,y) \big| ^{1/H} |\sigma\big(u(r,y)\big) |^{1/H} dy \right)^{pH} \right]dr\\
 & \le c_p C_H^{p/2}  T^{p/2-1}
  \int_0^T  \left( \int_{\R}  \big|\varphi_{t,R}(r,y)  - \varphi_{s,R}(r,y) \big| ^{1/H} \|\sigma\big(u(r,y)\big) \|_{p} ^{1/H} dy \right)^{pH} dr\\
  &\le c_p C_H^{p/2}  T^{p/2-1} K^p_p(T) 
   \int_0^T    \left( \int_{\R}  \big|\varphi_{t,R}(r,y)  - \varphi_{s,R}(r,y) \big| ^{1/H} dy \right)^{pH} dr.
   \end{align*}
 Finally,  from   \eqref{est11}, which holds true for any $R\geq T$, we can write
 \[
 \left( \int_{\R}  \big|\varphi_{t,R}(r,y)  - \varphi_{s,R}(r,y) \big| ^{1/H} dy \right)^{pH} 
 \le  4^{p(1+H)} (t-s)^p    R^{pH}.
 \]
  It is then straightforward to get \eqref{tendue}.
\end{proof}

\begin{proof}[Proof of Theorem \ref{thm:functional-CLT}]

We need to prove  tightness and the  convergence of the finite-dimensional distributions.   Notice that   tightness   follows   from Proposition \ref{pro:tightness} and the well-known criterion of Kolmogorov.

Let us now show the convergence of the finite-dimensional distributions. We fix   $0\le t_1< \cdots <t_m \le T$ and consider  
\[
F_R(t_i)  := \frac{1}{R^H} \left(\int_{-R}^R u(t_i,x)dx - 2R\right) = \delta\big( v_R^{(i)}\big) ~\text{for $i=1, \dots, m$,}
\]
where   
$$
v_R^{(i)} (s,y) =    \mathbf{1} _{[0,t_i]}(s) \frac{\sigma\big(u(s,y)\big)}{ R^H }\varphi^{(i)}_R(s,y) ~ \text{with}~     \varphi^{(i)}_R(s,y) =   \frac{1}{2} \int_{-R}^R  \mathbf{1}_{\{|x-y| \le t_i -s\}} \,  dx \,.
$$
Set $\mathbf{F}_R=\big( F_R(t_1), \dots, F_R(t_m) \big)$ and let $Z$ be a centered  Gaussian   vector on $\RR^m$ with covariance $(C_{i,j})_{1\leq i,j\leq m}$ given by  
\[
C_{i,j}    := \begin{cases} {\displaystyle 2 \int _0^{t_i \wedge t_j} (t_i-r) (t_j-r)  \xi(r)  ~dr}, \,\,\,\quad\quad\text{if $H=  1/2$;}\\
{\displaystyle 2^{2H} \int _0^{t_i \wedge t_j}   (t_i-r) (t_j-r)  \eta^2(r)  ~dr}, \quad\text{if $H\in (1/2,1)$. }
\end{cases}
\]
We recall here that $\xi(r)= \E\big[ \sigma^2\big(u(r,y)\big)\big]$ and $\eta(r)= \E\big[\sigma\big(u(r,y)\big)\big]$.  Then, we need to show $\mathbf{F}_R$ converges in distribution to $Z$ and   in view of  Lemma \ref{lemma: NP 6.1.2}, it suffices to show that for each $i,j$,   $\langle DF_R(t_i), v_R^{(j)} \rangle_{\HH}$ converges to  $C_{i,j} $   in $L^2(\Omega)$, as $R\to+\infty$.  The case $i=j$ has been tackled before and the other case        can be dealt with by   using   arguments similar  to  those in the proof of Theorem \ref{thm:TV-distance}. For the convenience of readers, we only sketch these arguments  as follows.

We consider two cases: $H=1/2$ and $H\in(1/2, 1)$. In each case, we need to show 
(i) $\E\big[  F_R(t_i) F_R(t_j)  \big] \to C_{i,j}$  and  (ii) ${\rm Var}\big( \langle DF_R(t_i), v_R^{(j)} \rangle_{\HH} \big)\to 0$, as $R\to+\infty$.
Point (i)   has been established in {\it Remark} \ref{rem2}.  To see point (ii) for the case $H=1/2$, we begin with the decomposition $\langle DF_R(t_i), v_R^{(j)} \rangle_{\HH}  = B_1(i,j) +B_2(i,j)  $ with
         \begin{align*}
         B_1(i,j): =  \frac{1}{R}   \int_0^{t_i\wedge t_j} \int_{\R}  \varphi_R^{(i)}(s,y) \varphi_R^{(j)}(s,y)   \sigma ^2(u(s,y)) \,ds\, dy 
\end{align*}
and
\begin{align*}
B_2(i,j):&= \frac { 1 }{R} \int_0^{t_i\wedge t_j} \int_{\R} \varphi^{(j)}_R(s,y) \sigma(u(s,y))  \\
&\qquad  \times\left( \int_s^{t_i}\int_\R \varphi^{(i)}_R(r,z) \Sigma(r,z) D_{s,y} u(r,z) W(dr,dz) \right)  \,    dyds.
\end{align*}
Then using \eqref{FACT00} and going through the same lines as for the estimation of $A_1, A_2$,  we can get 
\begin{align*}
&\quad \sqrt{{\rm Var}\big( B_2(i,j) \big)  } \leq \frac{1}{R} \int_0^{t_i\wedge t_j} ds\Bigg(   \int_{\R^3}  \int_s^{t_i}   \varphi^{(i)}_R(r,z)^2  \varphi^{(j)}_R(s,y)   \varphi^{(j)}_R(s,y') \\
  &\qquad \times
     \E  \Big[    \Sigma^2(r,z) D_{s,y} u(r,z)  D_{s,y'} u(r,z) 
  \sigma\big(u(s,y)\big)\sigma\big(u(s,y')\big) \Big] \, dydy'    dz  dr
   \Bigg)^{  1/2}  \\
   &\leq \frac{C}  {R} \int_0^{t_i\wedge t_j} \Bigg(   \int_{\R^3}  \int_s^{t_i}   \varphi_R^{(i)}(r,z)^2   \varphi^{(j)}_R(s,y)   \varphi^{(j)}_R(s,y') \\
  &\qquad    \qquad\qquad\qquad\times
     \mathbf{1}_{\{|y-z| \vee |y' -z| \le r-s\}} dydy'    dz  dr
   \Bigg)^{  1/2} ds \leq \frac{C}  {\sqrt{R}}.
   \end{align*}
That is, we have ${\rm Var}\big( B_2(i,j) \big) \to 0$, as $R\to+\infty$. We can also get 
\begin{align*}
&\quad \sqrt{{\rm Var}\big( B_1(i,j) \big)  } \leq \frac{1}{R} \int_0^{t_i\wedge t_j}    \Bigg(   \int_{\R^2}  \varphi^{(i)}_R(s,y)  \varphi^{(j)}_R(s,y)  \varphi^{(i)}_R(s,y')\varphi^{(j)}_R(s,y')\\
&\qquad\qquad\qquad\qquad\qquad \qquad \quad \times {\rm Cov} \Big[ \sigma^2 \big(u(s,y)\big)  , \sigma ^2\big(u(s, y')\big)  \Big] dydy'   \Bigg)^{1/2} ds\\
   &\leq \frac{C}{R} \int_0^{t_i\wedge t_j}    \Bigg(   \int_{\R^2}  \varphi^{(i)}_R(s,y)  \varphi^{(j)}_R(s,y)  \varphi^{(i)}_R(s,y')\varphi^{(j)}_R(s,y') \mathbf{1}_{\{ \vert  y - y' \vert \leq 2s \}} dydy'   \Bigg)^{1/2} ds \\
   &\leq \frac{C}{R} \int_0^{t_i\wedge t_j}    \Bigg(   \int_{\R^2} (t_i + t_j)^4  \mathbf{1}_{\{ \vert y\vert \vee \vert y'\vert \leq R + t_i + t_j \}}     \mathbf{1}_{\{ \vert  y - y' \vert \leq 2s \}} dydy'   \Bigg)^{1/2} ds \leq \frac{C}{\sqrt{R}} \,.
   \end{align*}
That is, we have ${\rm Var}\big( B_1(i,j) \big) \to 0$, as $R\to+\infty$.  

To see point (ii) for the case $H\in(1/2, 1)$,  one can begin with the same decomposition and then use \eqref{FACT00} to arrive at similar estimations as those for $\mathbf{I}$ and $\mathbf{J}$. Therefore the same arguments  ensure
$
{\rm Var}\big(\langle DF_R(t_i), v_R^{(j)} \rangle_{\HH}\big)  \leq  CR^{2H-2}
$.     Now the proof of Theorem \ref{thm:functional-CLT}
 is completed. \qedhere
\end{proof}

 \section{Appendix: Proof of  Lemma \ref{lemma: iteration} }
  
  This appendix provides the proof of our technical Lemma and it consists of two parts.  The first part proceeds assuming 
  \begin{align}\label{claimbdd}
\mathfrak{L}:= \sup_{(r,z)\in [0,t]\times \RR} \| D_{s,y} u(r, z) \|_p < +\infty ~\text{for almost every $(s,y)\in\R_+\times\R$}
       \end{align}
and the second part is devoted to establishing  the above bound. Note that \emph{a priori}, we do not know whether $D_{s,y}u(r,z)$ is a function of $(s,y)$ or not in the case where $H\in(1/2, 1)$, so the assumption \eqref{claimbdd} also guarantees that  $D_{s,y}u(r,z)$ is indeed a random function in $(s,y)$; see Section \ref{Sec52} for more explanation.

\subsection{Proof of Lemma \ref{lemma: iteration}  assuming \eqref{claimbdd}}
The proof will be done in two steps.

\medskip
\noindent
 {\it Step 1: Case  $H= 1/2$}.    ~     From (\ref{ecu1}), using Burkholder's and Minkowski's inequality,  we can write
   \begin{align*}  
 &\quad \| D_{s,y}u(t,x) \|_p \\
 & \leq \frac{ K_{p}(t)}{2}  \mathbf{1}_{\{ |x-y| \le t-s\}}      +   \frac{Lc_p}{2}  \left( \int_s^t \int_{\R}  \mathbf{1}_{\{ |x-z| \le t-r \}}    \| D_{s,y} u(r,z) \|_p^2 drdz \right)^{1/2}
\end{align*}
 with   $c_p$ a constant that only depends on $p$. It follows from the elementary inequality $(a+b)^2\leq 2a^2+2b^2$ that 
  \begin{align*}  
 \| D_{s,y}u(t,x) \|_p^2 \le \frac{K^2_{p}(t)}{2}  \mathbf{1}_{\{ |x-y| \le t-s\}}      +   \frac{ L^2c^2_p}{2}  \int_s^t \int_{\R}  \mathbf{1}_{\{ |x-z| \le t-r \}}    \| D_{s,y} u(r,z) \|^2_p drdz.
\end{align*}
Iterating this inequality yields, for any positive integer $M$,
   \begin{align}  
  &\quad \| D_{s,y}u(t,x) \|_p^2    \label{tt1}      \leq  \frac { K^2_{p}(t) } 2 \mathbf{1}_{\{ |x-y| \le t-s\}}  \notag  \\
   &     +\frac { K^2_{p}(t) } 2   \sum_{N=1}^M   \frac {  c_p^{2N} L^{2N}} {2^N}   \int_{\Delta_N(s,t)} \int_{\R^{N}} 
\left(\prod_{n=1} ^{N}  \mathbf{1}_{\{ |z_{n-1}-z_n| \le r_{n-1}-r_n \}} \right)   \mathbf{1}_{\{ |z_{N}-y| \le r_{N}-s \}}      d{\bf r}d{\bf z}  \notag \\
&  +  \frac {  c_p^{2M+2} L^{2M+2}} {2^{M+1}}   \int_{\Delta_{M+1}(s,t)} \int_{\R^{M+1}} 
\left(\prod_{n=1} ^{M}  \mathbf{1}_{\{ |z_{n-1}-z_n| \le r_{n-1}-r_n \}} \right)   \mathbf{1}_{\{ |z_{M}- z_{M+1}| \le r_{M}- r_{M+1} \}}     \notag \\
&\qquad\qquad\qquad\qquad\qquad\qquad\qquad\qquad\qquad \times  \| D_{s,y} u(r_{M+1}, z_{M+1}) \|^p_2\,   d{\bf r}d{\bf z} \,,   \notag
\end{align} 
where  $\Delta_N(s,t) := \big\{ ( r_1,  \dots, r_N)\in\RR^N\vert  s<r_N < r_{N-1} < \cdots <r_1<t\big\}$, $d{\bf r}= dr_1 \cdots dr_{N}$,
$d{\bf z} = dz_1 \cdots dz_{N}$ and with the convention  $r_0=t$ and   $z_{0}=x$. 

Notice that if
\[
\left(\prod_{n=1} ^{N}  \mathbf{1}_{\{ |z_{n-1}-z_n| \le r_{n-1}-r_n \}} \right)  \mathbf{1}_{\{ |z_{N}-y| \le r_{N}-s \}}  \not =0,
\]
 then on $\Delta_N(s,t)$,
$
|x-y| = |z_{0} -y |  \le  \sum_{n=1} ^{N} |z_{n-1} -z_n| + |z_N-y|\le   t-s
$
and similarly on $\Delta_N(s,t)$, $| z_n-y| \le  t-s$ for\footnote{This in particular implies that the contribution of the integration with respect to $dz_n$ is at most $2(t-s)$.} $n=1, \dots, N$. 

Now we deduce from \eqref{claimbdd} that
\begin{eqnarray*}
\| D_{s,y}u(t,x) \|_p^2  &\le&   \mathbf{1}_{\{ |x-y| \le t-s\}} \frac { K^2_{p}(t) }2   \left(
 1 +  \sum_{N=1} ^\infty     c_p^{2N} L^{2N}   
 \frac {(t-s)^{2N} }{N!}      \right)\\
 &\le & \mathbf{1}_{\{ |x-y| \le t-s\}} \frac { K^2_{p}(t) }2  \exp \big(  c_p^2 L^2 t^2\big),
\end{eqnarray*}
which provides the desired estimate. 
  
\medskip
\noindent
{\it  Step 2: Case  $H\in( 1/2, 1)$}.   ~      Proceeding as before,  and using the inequality
  \[
  \| D_{s,y} u(r,z)D_{s,y} u(r,\tilde{z}) \|_{p/2} 
  \le \frac 12 \left( \| D_{s,y} u(r,z) \|_p^2 + \| D_{s,y} u(r,\tilde{z}) \|_p^2 \right),
  \]  
  we obtain
  \begin{align*}  
 &\| D_{s,y}u(t,x) \|_p^2 \le \frac{ K^2_{p}(t) } 2 \mathbf{1}_{\{ |x-y| \le t-s\}}    \\     
& \quad  +   \frac {   c^2_p L^2} 2  \alpha_H \int_s^t \int_{\R^2}  \mathbf{1}_{\{ |x-z| \le t-r \}} 
\mathbf{1}_{\{ |x-\tilde{z}| \le t-r \}}    \| D_{s,y} u(r,z) \|^2_p |z- \tilde{z} | ^{2H-2}  drdz d\tilde{z} .
\end{align*}
By iteration, this leads to the following estimate. For any positive integer $M$,
 \begin{align*}  
 &\| D_{s,y}u(t,x) \|_p^2 \le \frac { K^2_{p}(t) } 2 \mathbf{1}_{\{ |x-y| \le t-s\}}    +  \frac { K^2_{p}(t) } 2 \sum_{N=1} ^M   \frac {  c_p^{2N} L^{2N}} {2^N}  \int_{\Delta_N(s,t)}  \alpha_H^N \int_{\R^{2N}}   \\
 &  \quad
\left(\prod_{n=1} ^{N}   
 \mathbf{1}_{\{ |z_{n-1}-\tilde{z}_n| \vee|z_{n-1}-z_n|   \le r_{n-1}-r_n \}}    |z_n- \tilde{z}_n | ^{2H-2}  \right)    \mathbf{1}_{\{ |z_{N}-y| \le r_{N}-s \}}   d{\bf r}d{\bf z}d{\bf \tilde{z}} \\
 &+ \frac{(Lc_p)^{2M+2}}{2^{M+1}} \int_{\Delta_{M+1}(s,t)} d\mathbf{r} ~ \alpha_H^{M+1} \int_{\R^{2M+2}} d\mathbf{z}d\tilde{\mathbf{z}}    \\
 &  \times  \left( \prod_{n=1} ^{M+1}   
 \mathbf{1}_{\{ |z_{n-1}-\tilde{z}_n| \vee|z_{n-1}-z_n|   \le r_{n-1}-r_n \}}    |z_n- \tilde{z}_n | ^{2H-2}  \right)\| D_{s,y} u(r_{M+1}, z_{M+1})\|_p^2
\end{align*}
 with the same convention as before.   
  Note that Lemma \ref{lem1} implies that on $\Delta_N(s,t)$, 
 \begin{align}
 &  \alpha_H \int_{\RR^2}  \mathbf{1}_{\{  \vert z_{n-1} - z_{n} \vert \vee \vert z_{n-1} - \tilde{z}_{n} \vert  \leq r_{n-1} - r_{n}  \}}  \vert z_{n} - \tilde{z}_n \vert^{2H-2}  dz_{n}d\tilde{z}_n \notag \\
 &\leq \alpha_H \int_{\RR^2}  \mathbf{1}_{\{  \vert z_{n-1} - z \vert \vee \vert z_{n-1} - z' \vert  \leq t \}}  \vert z - z' \vert^{2H-2}  dzdz' \leq 4^H t^{2H} \,; \label{ele3}
 \end{align}
 and note that we again have the following implication:
 \[
\prod_{n=1} ^{N}  \mathbf{1}_{\{ |z_{n-1}-z_n| \le r_{n-1}-r_n \}}   \mathbf{1}_{\{ |z_{N}-y| \le r_{N}-s \}}  \neq0 \Longrightarrow \vert x - y\vert  \leq t -s \,,
\]
which, together with \eqref{ele3}, implies
 \begin{align*}  
 \| D_{s,y}u(t,x) \|_p^2& \le \frac { K^2_{p}(t) } 2 \mathbf{1}_{\{ |x-y| \le t-s\}}    +  \frac { K^2_{p}(t) }{2} \mathbf{1}_{\{ |x-y| \le t-s\}}  \sum_{N=1} ^M   \frac {  (Lc_p 2^H t^{H})^{2N} } {2^N}  \frac{t^N}{N!} \\
 &  \qquad  + \mathfrak{L}^2 ~ \frac {  (Lc_p 2^H t^{H})^{2M+2} } {2^{M+1}}  \frac{t^{M+1}}{(M+1)!} \qquad  \text{($\mathfrak{L}$ is defined in \eqref{claimbdd})} \,.
\end{align*}
 Letting $M\to+\infty$ leads to 
 \begin{align*}
  \| D_{s,y}u(t,x) \|_p^2 & \leq \frac{ K^2_{p}(t) }{2} \mathbf{1}_{\{ |x-y| \le t-s\}}    +  \frac { K^2_{p}(t) }{ 2}\mathbf{1}_{\{ |x-y| \le t-s\}}  \sum_{N=1} ^\infty  \frac {  (Lc_p 2^H t^{H})^{2N} } {2^N}  \frac{t^N}{N!} \\
  &\leq  \frac { K^2_{p}(t) }{2}\exp\big( 2L^2c_p^2 t^{2H+1} \big) \mathbf{1}_{\{ |x-y| \le t-s\}}   \,.
 \end{align*}
This concludes our proof of Lemma \ref{lemma: iteration} assuming \eqref{claimbdd}.
\medskip

\subsection{Proof of \eqref{claimbdd}}\label{Sec52}
 The proof will be done in two steps.

\medskip
\noindent
{\it Step 1: Case  $H=1/2$}.~  It  is well known  in the literature that 
for any $p\ge 2$, 
  $u(t,x)\in\mathbb{D}^{1,p}$ and 
\begin{align}\label{shallimp}
\sup_{(t,x)\in[0,T]\times \RR } \E\big[ \| Du(t,x)\| _{\mathfrak{H}}^p \big] < +\infty \,;
\end{align}
indeed,  in the Picard iteration scheme (see \emph{e.g.} \eqref{Picardit}), one can first prove the iteration $u_n$ converges to the solution $u$ in $L^p(\Omega)$ uniformly in $[0,T]\times \R$, then we   derive the uniform bounded for $\E\big[ \| Du_n(t,x) \| _\mathfrak{H}^p\big]$, so that by standard Malliavin calculus argument, we can get the convergence of $Du_n(t,x)$ to $Du(t,x)$ with respect to the weak topology on $L^p(\Omega ; \mathfrak{H})$ and hence the desired uniform bound \eqref{shallimp}. We omit the details for this case ($H=1/2$) and refer to the arguments for the other case ($H>1/2$).

Consider an approximation of the identity $\big( M_\varepsilon,\varepsilon > 0\big) $   in $L^1(\RR_+\times\RR)$ satisfying
  $M_\varepsilon(s,y) = \varepsilon^{-2} M(s/\varepsilon,y/\varepsilon)$ for some nonnegative   $M\in   C_c(\R_+\times\R)$.
Taking into account that $(\omega, s,y) \to D_{s,y}u(t,x)$ belongs to $L^2( \R_+\times\R ; L^2(\Omega))$, we deduce that
the convolution  $Du(t,x) * M_\varepsilon$ converges to  $ Du(t,x)$ in $L^2( \R_+\times\R ; L^2(\Omega))$, as $\varepsilon$ tends to zero.
Therefore, there exist a sequence $\{\varepsilon_n\}$ such that   $\varepsilon_n \downarrow 0$ and  $(Du(t,x) * M_{\varepsilon_n})(s,y)$  converges almost surely to $D_{s,y}u(t,x)$ for almost all  $(s,y) \in  \R_+\times \R$, as $n\to+\infty$. By Fatou's Lemma, this implies that for almost all $(s,y) \in \R_+\times\R$, 
\begin{equation}
\| D_{s,y} u(t,x) \| _p \le  \sup_{n\in\N} \big\| \big(Du(t,x) \ast M_{\varepsilon_{n}} \big)(s,y) \big\| _p \label{further0}.
\end{equation} 

  Now we fix $(s,y)$ that satisfies \eqref{further0}  and put for $\varepsilon > 0$
  \begin{align} 
Q_\varepsilon(t) :& = \sup_{z\in\R}  \big\| \big(Du(t,z) \ast M_{\varepsilon} \big)(s,y) \big\| _p^2    \label{Q(t)}\\
&= \sup_{z\in\R}   \left\|  \int_{\R_+\times \R}   D_{s',y'} u(t, z)    M_{\varepsilon}(s'-s, y'-y) \, ds'dy' \right\|_p^2\,,~ t\in [0,T]. \notag
  \end{align}
 In the following,
  \begin{enumerate} 
  \item we will prove for each $\varepsilon > 0$, $Q_\varepsilon$ is uniformly bounded on $[0,T]$; 
  \item we  will obtain an integral inequality for $Q_\varepsilon$;
  \item  we will conclude with the classic Gronwall's lemma.
  \end{enumerate}
  Recall from \eqref{ecu1} and we can write 
  \begin{align*}
&\big(Du(t,z) \ast M_{\varepsilon} \big)(s,y) = \frac{1}{2} \int_{\R_+\times \R}{\bf 1}_{\{  | z - y' | \leq t -s' \}} \sigma\big(u(s',y') \big)M_{\varepsilon}(s'-s, y'-y)  ds'dy' \\
&\qquad\qquad   +  \frac{1}{2} \int_0^t\int_\R {\bf 1}_{\{  | z - \xi | \leq t -a \}}\Sigma(a,\xi)  \big( Du(a,\xi)\ast M_{\varepsilon}\big)(s,y)    W(da,d\xi)  .  
  \end{align*}
    Then, using   Burkholder's inequality  and Minkowski's inequality in the same way as before,  we can arrive at 
  \begin{align*}
   Q_\varepsilon(t) & \leq \frac{K_p^2(t)}{2} \left(\sup_{z\in\R} \int_{\R_+\times \R} \mathbf{1}_{\{ \vert z - y'\vert \leq t-s' \}}  M_{\varepsilon}(s'-s, y'-y) \, ds'dy' \right)^2 \\
    &\qquad + L^2c_p^2 t \int_0^t Q_\varepsilon(a)\, da \\
  &\leq   \frac{K_p^2(t)}{2} \| M \| _{L^1(\R_+\times\R)} ^2 + L^2c_p^2 t \int_0^t Q_\varepsilon(a)\, da, ~ t\in[0,T].
  \end{align*}
  We know from \eqref{shallimp}  and Cauchy-Schwarz inequality that 
  \[
  Q_\varepsilon(t) \leq \| M_\varepsilon\|_\mathfrak{H}^2 \sup_{z\in\R} \Big( \E\big[ \| Du(t,z)\|_\mathfrak{H}^p \big] \Big)^{2/p} \,,
  \]
  which is uniformly bounded on $[0,T]$.   Then it follows from  Gronwall's lemma that 
  \[
   Q_\varepsilon(t) \leq  \frac{K_p^2(t)}{2}e^{L^2c_p^2 t}   \| M \| _{L^1(\R_+\times\R)} ^2, ~ \forall t\in[0,T].
  \]
The above bound is independent of $\varepsilon$, thus  we can further deduce that 
  \[
   \sup_{(r,z)\in[0,t]\times\R}\| D_{s,y} u(r, z) \|_p^2 \leq  \frac{K_p^2(t)}{2}e^{L^2c_p^2 t}   \| M \| _{L^1(\R_+\times\R)} ^2 < +\infty \,.
  \]
  That is, claim \eqref{claimbdd} is established  for the case $H=1/2$.\\

  \medskip
  \noindent
  {\it   Step 2: Case  $H\in(1/2,1)$}. ~ 
  In this case we have first to show that $Du(t,x)$ is an element of $L^2( \Omega \times \R_+ \times \R)$ and for this we will use the Picard iterations. 
  Let $u_0(t,x) =1$ and  for $n\ge 0$, set 
  \begin{align}
  u_{n+1}(t,x) = 1 + \frac{1}{2} \int_0^t \int_\R \mathbf{1}_{\{ \vert x - y \vert \leq t-s\}} \sigma\big( u_n(s,y) \big) \, W(ds, dy) \,. \label{Picardit}
  \end{align}
 It is routine to show that  for any given $T\in\R_+$,
 \begin{align}\label{UNILP}
 \lim_{n\to+\infty} \sup_{(t,x)\in[0,T]\times\RR }      \| u(t,x) - u_{n}(t,x) \|_p    =0 \,.
 \end{align}
                  We know that   
 for each $n\geq 0$, $u_n(t,x)\in \mathbb{D}^{1,p}$ with
   \begin{align*}  
&\quad D_{s,y}u_{n+1}(t,x)=        \frac{1}{2}  \mathbf{1}_{\{ \vert x - y \vert \leq t - s \} } \sigma\big( u_{n}(s,y) \big)   \\
&\qquad\qquad\qquad\qquad \qquad +  \frac{1}{2} \int_s^t \int_\RR \mathbf{1}_{\{ \vert x - z \vert \leq t - r \} }  \Sigma_n(r,z)   D_{s,y} u_n(r,z)  \, W(dr, dz) \,,
\end{align*}
with $\Sigma_n(r,z)  $ being an adapted process bounded by $L$.
Thus, using Burkholder's inequality, Minkowski's inequality and 
 the easy inequality 
 \begin{align}
 \| XY\|_{p/2} \leq \frac{1}{2} \| X \| _p^2 +\frac{1}{2} \| Y \| _p^2  \label{ele2}
 \end{align} 
 for any $X, Y\in L^p(\Omega)$,  we   get $ \| D_{s,y}u_{n+1}(t,x) \|_p^2 $ bounded by 
   \begin{align*}  
  \frac{ \widetilde{K}^2_{p}(t)  }{2}  \mathbf{1}_{\{ |x-y| \le t-s\}}      +   \frac{L^2c_p^2}{2}   \int_s^t  \alpha_H \int_{\RR^2}  & \mathbf{1}_{\{ |x-z| \vee \vert x - z' \vert  \le t-r \}} \vert z- z'\vert^{2H-2}  \\
  &\qquad\qquad\qquad  \times \| D_{s,y} u_n(r,z) \|_p^2 ~ drdzdz' \,,
\end{align*}
 where $\widetilde{K}_p(t):=\sup\big\{  \| \sigma ( u_n(s,x)  )  \| _p \,: n\geq 0,  (s,x)\in[0, t]\times \RR\big\} $.   Iterating this procedure gives us
 \begin{align*}   
  &  \| D_{s,y}u_{n+1}(t,x) \|_p^2      \leq  \frac { \widetilde{K}^2_{p}(t)  }{2} \mathbf{1}_{\{ |x-y| \le t-s\}}       +\frac {\widetilde{K}^2_{p}(t)  }{2}   \sum_{\ell=1}^n   \frac {  c_p^{2\ell} L^{2\ell}} {2^\ell}  \int_{\Delta_\ell(s,t)}    d{\bf r}  \alpha_H^\ell \int_{\R^{2\ell}} \\
  &  \quad\times 
\left(\prod_{k=0} ^{\ell-1}  \mathbf{1}_{\{ |z_{k}-z_{k+1}|  \vee |z_k-z_{k+1}'|    \le r_{k-1}-r_k \}}  \vert z_{k+1} - z_{k+1}' \vert^{2H-2} \right)   \mathbf{1}_{\{ |z_{\ell}-y| \le r_{\ell}-s \}}     d{\bf z}' d{\bf z}    \,.
\end{align*} 
 Again, it is easy to see the following implication holds:
   $$  \mathbf{1}_{\{ |z_{\ell}-y| \le r_{\ell}-s \}}     \prod_{k=0} ^{\ell-1}  \mathbf{1}_{\{ |z_{k}-z_{k+1}|  \vee |z_k-z_{k+1}'|    \le r_{k-1}-r_k \}} \neq 0 \Longrightarrow \vert x - y\vert  \leq t-s \,,$$   
 therefore
  \begin{align*}   
  &  \| D_{s,y}u_{n+1}(t,x) \|_p^2      \leq  \frac { \widetilde{K}^2_{p}(t)  }{2} \mathbf{1}_{\{ |x-y| \le t-s\}}       +\frac {\widetilde{K}^2_{p}(t)  }{2}  \mathbf{1}_{\{ |x-y| \le t-s \}}  \sum_{\ell=1}^n   \frac {  c_p^{2\ell} L^{2\ell}} {2^\ell}  \int_{\Delta_\ell(s,t)}    d{\bf r}  \\
  & \qquad\qquad  \times  \alpha_H^\ell \int_{\R^{2\ell}}
\left(\prod_{k=0} ^{\ell-1}  \mathbf{1}_{\{ |z_{k}-z_{k+1}|  \vee |z_k-z_{k+1}'|    \le r_{k-1}-r_k \}}  \vert z_{k+1} - z_{k+1}' \vert^{2H-2} \right)       d{\bf z}' d{\bf z}   \\
&\leq    \frac { \widetilde{K}^2_{p}(t)  }{2} \mathbf{1}_{\{ |x-y| \le t-s\}}       +\frac {\widetilde{K}^2_{p}(t)  }{2}  \mathbf{1}_{\{ |x-y| \le t-s \}}  \sum_{\ell=1}^n   \frac {  c_p^{2\ell} L^{2\ell}} {2^\ell} (4^H t^{2H+1})^\ell  \frac{1}{\ell!},
\end{align*} 
 where the last inequality is  a consequence of \eqref{ele3}. We conclude that
 \begin{align}
 \| D_{s,y}u_{n+1}(t,x) \|_p^2   \leq   \frac { \widetilde{K}^2_{p}(t)  }{2} e^{ 2 t^{2H+1} c_p^2 L^2 }  \mathbf{1}_{\{ |x-y| \le t-s\}}   =: C   \mathbf{1}_{\{ |x-y| \le t-s\}} . \label{DnBDD}
 \end{align}
 It follows immediately from Minskowski's inequality and \eqref{DnBDD}    that 
 \[
 \E\big[ \| Du_n(t,x) \| _{L^2(\R_+\times\R)}^p  \big] \leq  \left( \int_{\R_+\times\R} \| D_{s,y} u_n(t,x) \| _p^2 \, dsdy \right)^{p/2} \leq  \big( C t^2 \big)^{p/2}
 \]
{\it uniformly} in $n\geq 1$ and {\it uniformly} in $x\in\R$. In particular, $\big\{ Du_n(t,x) , n\geq 1 \big\}$ is uniformly bounded in $L^p(\Omega ; L^2(\R_+\times\R))$. Note that the   convergence in  \eqref{UNILP} and  standard Malliavin calculus arguments  can lead us to the fact that up to some subsequence, $Du_n(t,x)$ converges to $Du(t,x)$ in the weak topology of $L^p(\Omega ; L^2(\R_+\times\R))$, so we can conclude that $D_{s,y}u(t,x)$ is indeed a function in $(s,y)$ and for any fixed $T \in\R_+$,
\[
\sup_{(t,x)\in[0,T]\times\R}  \E\big[ \| Du(t,x) \| _{L^2(\R_+\times\R)}^p  \big] < +\infty \,.
\]
  Now we   use the same approximation of the identity $( M_\varepsilon )$ and obtain for almost every $(s,y)\in\R_+\times\R$,
\[
\| D_{s,y} u(t,x) \| _p^2 \leq \sup_{\varepsilon > 0}   \left\|  \int_{\R_+\times \R}   D_{s',y'} u(t, x)    M_{\varepsilon}(s'-s, y'-y) \, ds'dy' \right\|_p^2 \,.
\]
 Let $\varepsilon > 0$ be fixed and let $Q_\varepsilon(t)$ be defined as in \eqref{Q(t)}, we have in this case, applying  Lemma \ref{lem1},
\begin{align*}
 Q_\varepsilon(t) 
&\leq \frac{1}{2} K^2_p(t) \| M \|^2_{L^1(\R_+\times\R)} \\
&  \quad + \frac{L^2c_p^2}{2}  \int_0^t\alpha_H\int_{\R^2} \mathbf{1}_{\{ \vert x -z \vert\vee \vert x-z'\vert\leq t-r\}} \vert z-z'\vert^{2H-2} Q_\varepsilon(r)\, dr dzdz'  \\
&\leq \frac{1}{2} K^2_p(t) \| M \|^2_{L^1(\R_+\times\R)} +  \frac{L^2c_p^2 4^H t^{2H}}{2} \int_0^tQ_\varepsilon(r)\, dr.
\end{align*} 
Similarly as in previous case, we have 
\[
Q_\varepsilon(t) \leq \| M_\varepsilon \| _{L^2(\R_+\times\R)}^2 \Big( \E\big[ \| Du(t,x) \| _{L^2(\R_+\times\R)}^p  \big]  \Big)^{2/p}
\]
 so that the same application of Gronwall's lemma gives 
  \[
   \sup_{(r,z)\in[0,t]\times\R}\| D_{s,y} u(r, z) \|_p^2 \leq  \frac{K_p^2(t)}{2}e^{2L^2c_p^2 t^{2H}}   \| M \| _{L^1(\R_+\times\R)} ^2 < +\infty \,.
  \]
  That is, claim \eqref{claimbdd} is also established  for the case $H\in(1/2,1)$.\hfill $\square$

 \medskip
 \noindent{\bf Acknowledgement:} We would like to thank two anonymous referees for their helpful remarks and  in particular to one of them for detecting a mistake in the proof of Theorem
 \ref{thm:TV-distance} and for providing a generous amount of comments that improve our paper.

\end{document}